\documentclass[11pt]{article}
\usepackage{amsmath,amsthm,amsfonts,amssymb,bm}
\usepackage{bbm}
\usepackage{mathtools}
\usepackage{graphicx}
\usepackage{microtype}
\usepackage{lmodern}
\usepackage{arydshln}
\usepackage{yfonts}
\usepackage{breqn}
\usepackage[final]{showkeys}
\usepackage{caption}
\usepackage{geometry}
\numberwithin{equation}{section}

\newtheorem{teo}{Theorem}[section]
\newtheorem{prop}{Proposition}[section]
\newtheorem{lem}{Lemma}[section]
\newtheorem{defn}{Definition}[section]
\newtheorem{obs}{Remark}[section]
\newtheorem{cor}{Corollary}[section]

\def\bbR{{\mathbb R}}
\def\bbN{{\mathbb N}}

\def\bbE{{\mathbb E}}
\def\bbP{{\mathbb P}}

\def\bbR{{\mathbb R}}

\def\bbM{{\mathcal M}}
\def\bbF{{\mathcal F}}
\def\bbG{{\mathcal G}}

\def\bbT{{\mathcal T}}

\title{On the continuous-time limit of the Barab{\'a}si--Albert random graph}

\author{Angelica Pachon$^{1}$, Federico Polito$^{2}$ \& Laura Sacerdote$^{2}$\\
	\footnotesize{${}^1$Faculty of Computing, Engineering and Science, University of South Wales, UK}\\
	\footnotesize{${}^2$Mathematics Department ``G.~Peano'', University of Torino, Italy}}

\begin{document}

	\maketitle

	\abstract

		We prove that, via an appropriate scaling,
		the degree of a fixed vertex in the Barab\'asi--Albert model appeared at a large enough time converges in
		distribution to a Yule process. 
		Using this relation we explain why the limit degree distribution of a vertex chosen uniformly at random
		(as the number of vertices goes to infinity),
		coincides with the limit distribution of the number of species in a genus selected uniformly
		at random in a Yule model (as time goes to infinity).
		To prove this result we do not assume that the number
		of vertices increases exponentially over time (linear rates).
		On the contrary, we retain their natural growth with a constant rate
		superimposing to the overall graph structure a suitable set of processes that we call the \emph{planted model} and
		introducing an ad-hoc sampling procedure.
		
		\medskip
		
		\noindent \textit{Keywords}: Barab\'asi--Albert model; Preferential attachment random graphs;
			Planted model; Discrete- and continuous-time models; Yule model.

		\medskip		
		
		\noindent \textit{MSC2010}: 05C80, 90B15, 60J80.

	\section{Introduction}
	
		One of the most popular models for network growth is the preferential attachment model proposed by
		Barab\'asi and Albert \cite{Barabasi1999} to describe the web graph growth.
		In this model a newly created vertex is connected to one of
		those already present in the graph with a probability proportional to their degrees.
		An important characteristic related to a non-degenerate
		preferential attachment growth mechanism is the presence of a power-law distribution for the asymptotic degree of
		a vertex selected uniformly	at random.
		This property is actually observed on World Wide Web data
		\cite{Barabasi1999,Faloutsos:1999,Kleinberg:1999}.
		Furthermore, power-law distributions also occur frequently in other real-world phenomena and many of them are strictly related
		to the preferential attachment paradigm	\cite{2009SIAMR,342,343,2011Natur,Newmannew}. This fact
		determines an increasing interest on the Barab\'asi--Albert model (BA model in the following) and for random graphs
		growing with preferential attachment rules in general.
		Indeed, there is an already extensive literature analyzing this class of random graphs. See for recent references Chapter 8 in \cite{hofstad2016} and the papers cited therein; we also recall \cite{PhysRevLett.85.4633,redner,PhysRevLett.85.4629}. The typical techniques considered, also implemented in the latter papers, are mainly of combinatorial type and based on the analysis of the expectation of specific functions of the degree or in-degree,
		together with concentration inequalities \cite{Bollobas2001,CooperFrieze,Remco92}. Other methods involve continuum and discrete approaches
		to study large but finite growing graphs \cite{redner2} and an embedding of the random graph processes into a continuous setting
		involving a sequence of pure birth continuous Markov chains (see \cite{Athreya,RSA:RSA20137,Bhamidi}). The technique of embedding
		a discrete sequence of random variables in continuous time processes is known for almost fifty years. When it is used  on
		random graph processes,  asymptotic results about properties of the vertices are obtained through an efficient use of branching
		process methods (see e.g. \cite{RSA:RSA20137,Bhamidi}).
		Despite its generality, the application of this technique  is not straightforward when the considered graph
		corresponds to the Barab\'asi--Albert model whose growth allows the simultaneous birth of $m\geq1$ new links. To deal with this problem
		it would be necessary either to develop suitable ``ad-hoc'' coupling techniques or merge vertices.
		For this last procedure see \cite{Bhamidi} for preferential attachment networks. 
		In continuous time there are other well-known probabilistic models which are clearly related to preferential attachment random graphs. Among them the Yule model for
		macroevolution \cite{Yule1925}.
		The relationship linking the Yule model with some discrete-time preferential attachment models is not
		straightforward \cite{PachonPolitoSacerdote} but can be exploited to effectively study discrete-time preferential attachment random graphs. In \cite{PachonPolitoSacerdote} we showed how a specific discrete-time model with preferential attachment, the Simon model \cite{SIMON1955}, is related
		(in a sense of weak convergence) to a set of Yule models. 
		
		The aim of this paper is to analyze specific aspects of the limiting behaviour of the BA model. To pursue this goal we follow and improve the methodology of \cite{PachonPolitoSacerdote} and relate the BA and the Yule models. Specifically, we couple the degree growth process of a vertex to a set of Markov processes,
		and we introduce the planted tree, an auxiliary branching structure superimposed to the random graph.		
		Then, by means of this, we establish that the degree of a vertex chosen uniformly at random converges in distribution
		to the size of a uniformly chosen genus in a $m$-Yule model (a Yule model characterized by an infinite sequence of independent Yule processes, each starting with $m$ individuals).
		We underline that we do not describe the dynamics of the degree of fixed vertices and the dynamics of the growth
		of the number of vertices with a given degree at the same time. Instead, we create a separate mechanism to describe the dynamics
		of the growth of the number of vertices with a given degree
		that does not need the Markov property of the degree processes. 
		
		We underline how this approach may prove to be useful in other
		cases as well, for example when the embedding method does not apply.
		Embedding techniques are problematic for more general preferential attachment models which are non-Markov,
		i.e., in which the emergence of future connections to existing vertices does not depend solely on the present state of their degree
		(for instance the connections could be affected by time delays of
		the random intervals at which the degree of a vertex changes, see \cite{BogunaSerrano} and the references therein).
		Other cases in which the approach of \cite{PachonPolitoSacerdote} and this paper might be applied are models with more general
		preferential attachment functions, such as models involving individual fitness \cite{Borgs:2007:FME:1250790.1250812} and/or
		aging \cite{GaravagliaHofstad}, but also models in which hybrid rules are considered such as the uniform/preferential attachment
		\cite{yang}.
				
		The	paper is organized as follows. The BA, the Yule and the $m$-Yule models are described in Section \ref{definizioni}.  
		In Section \ref{mainresults}, the main results are presented (the related proofs are contained Section \ref{proofs})
		together with a heuristic motivation of their validity.
		Summarizing the main results briefly, Theorem \ref{Teo1a} shows
		that when  infinitely many vertices have already appeared in the BA model, the degree
		distribution of a vertex appearing subsequently coincides with the distribution of the number
		of individuals in a Yule process starting with $m$ initial individuals. Note that this result is consistent with the known
		related result in the case of preferential attachment trees, that is when $m=1$ (see for example \cite{RSA:RSA20137}).
		Theorem \ref{Teo1} proves the convergence to the same limit distribution of the degree of a vertex chosen uniformly at random
		in the BA model (when the number of vertices diverges) and of the size of a genus chosen uniformly at random in an $m$-Yule
		model when time goes to infinity. We also prove that in the BA model the proportion of vertices with a given degree $k$
		converges in probability (as the number of vertices diverges) to the probability that the degree of a vertex chosen
		uniformly at random is equal to $k$.   
		The exact form of the limit distribution of Theorem \ref{Teo1} is then given in Proposition \ref{Teo2}.
		Furthermore, the above results can be extended to preferential attachment random graphs for which Lemma \ref{lema1} holds. This is mentioned in Remarks \ref{cere} and \ref{independentModel}.
		In Section  \ref{monotone}, a method for the sampling procedure of a random vertex in a general random graph model
		is proposed together with the notion of planted model.
		This method is a key tool to prove our main results.  More specifically, 
		this procedure is used to prove the relation
		between a randomly selected vertex in the BA model and a genus
		chosen uniformly at random from one of the $m$-Yule models which in turn is chosen uniformly at random from the set of all $m$-Yule models present in the planted model.
		Finally, as recalled above, Section \ref{proofs} contains the proofs of the above-mentioned results and the necessary
		auxiliary lemmas.

	\section{Preliminaries}\label{definizioni}

		\subsection{The Barab\'asi-Albert model}\label{BAl}

			In \cite{Barabasi1999}, the preferential attachment paradigm was proposed for the first time to model the growth of the World Wide Web.
			To do so the authors introduced a random graph model in which the vertices were added to the graph one at a time and joined to a
			fixed number of existing vertices, selected with probability proportional to their degree. In such a model the vertices
			represented the web pages and the edges their links. 
			In \cite{Barabasi1999} the model is described as follows:
					
			\begin{quote}
				\textit{Starting with a small number $(m_0)$ of vertices,
				at every time step  add a new vertex with $m$ $(\leq m_0)$
				edges that link the new vertex to $m$ different vertices already present in the system.
				To incorporate preferential attachment,  assume that the probability  that a new vertex will be connected
				to a vertex $i$ depends on the connectivity $k_i$ of that vertex, so it would be equal to $k_i/\sum_{j}k_j$.
				Thus, after $t$ steps the model leads to a random network with $t+m_0$ vertices and $mt$ edges.} 
			\end{quote}

			The model was then defined in rigorous mathematical terms by Bollob\'as et al.\
			\cite{Bollobas2001}. However, in this paper we follow a large part of the literature in referring to the above model as the
			Barab\'asi--Albert model even though it should be more correctly named after the authors of \cite{Bollobas2001}.
			Here we recall their definition for the growth of the random graph process $(G_m^t)_{t\geq 1}$.
			\begin{defn}\label{BAdefinition}
				For each $m\geq 1$ and for  every $n\in \bbN$,
				the process  $(G_m^t)_{t\geq 1}$  is such that,
				\begin{enumerate}
					\item  at time $t=n(m+1)+1$   a new vertex $v_{n+1}$ is added; 
					\item  for $i=2,\dots,m+1$, at each time $t=n(m+1)+i$ an edge from $v_{n+1}$ to $v$ is added
						with $v$ chosen with the following probabilities:
						\begin{align}
							\label{BA}
							\mathbb{P}(v_{n+1} \longrightarrow v) =
							\begin{cases}
								\dfrac{d(v,t-1)}{2(mn+i-1)-1}, &  v\neq v_{n+1}, \\ \\
								\dfrac{d(v,t-1)+1}{2(mn+i-1)-1}, & v=v_{n+1}.
							\end{cases}
						\end{align}
				\end{enumerate}
				In \eqref{BA} $d(v,t)$ denotes the degree of the vertex $v$ in $G_m^t$. 
			\end{defn}
			We explicitly underline that $(G_m^t)_{t\geq 1}$ starts
			at time $t=1$  with a single vertex, $v_1$,
			without loops. However, since at time $t=2$ the only existing vertex is $v_1$, then a loop is produced.

		\subsection{The Yule model}
				
			In this section we recall a classical continuous-time stochastic process which will be proven in the following to be strictly related to the BA random graph described in the preceding section. To avoid misunderstandings, we denote here by
			$T\in\mathbb{R}^+$ the continuous-time variable, while $t\in\bbN^* = \{1,2,\cdots\}$ indicates the discrete time.
			
			The model we are concerned with was introduced in 1925 by Yule \cite{Yule1925} to describe the macroevolution of a population characterized by the presence of different genera and species belonging to them.
			In order to describe it we first recall the well-known definition of a Yule process,
			i.e.\ a linear pure birth process in continuous time. A Yule model will then be defined in terms of a collection
			of independent Yule processes of possibly different birth intensities.
			
			\begin{defn}
				A Yule process $\{N(T)\}_{T\geq0}$ is a counting process in continuous time with state space $\bbN^*$,
				having initial condition $N(0)=g$, $g\geq1$, almost surely and
				infinitesimal transition probabilities
				\begin{align}\label{BirtPro}
					\bbP(N(T+h)=k+\ell\mid N(T)=k)=
					\begin{cases}
						k\lambda h+o(h), & \ell=1,\\
						o(h), & \ell>1,\\
						1-k\lambda h+o(h), & \ell=0,
					\end{cases}
				\end{align}
				where $\lambda>0$ is the birth intensity and $h>0$.
			\end{defn}  
			
			This process describes the growth of the size of a population in which, during any short time interval of length $h$ each member
			has probability $\lambda h+o(h)$, independently one another, to create a new individual.
			Note that the probability of simultaneous births is $o(h)$.
			
			Yule \cite{Yule1925} proposes to use independent copies of this process
			to model the growth of the number of species
			belonging to each separate genus.
			In turn, the evolution of the appearing genera is modelled by a further Yule process characterized by a possibly different birth intensity, say $\beta$, and
			independent of the former.
			The stochastic process determined by the combination of these two types of Yule processes is now known as a Yule model:
			\begin{defn}
				A Yule model describes the growth of the number of genera and species according to the following rules: 
				\begin{enumerate}
					\item genera (each comprising a single species) appear as a Yule process $\{N_{\beta}(T)\}_{T\geq0}$ of parameter $\beta$
					with one genus at time $T=0$ almost surely;
					\item each time a new genus appears, a copy of a Yule process of parameter $\lambda$  with
					a single initial progenitor starts. Those copies are independent one another and of the process of appearance
					of genera. Each copy models the evolution of species belonging to the same genus.
				\end{enumerate}  
			\end{defn}      

			In this paper we also consider an $m$-Yule model (denoted by $\{Y^m_{\lambda,\beta}(T)\}_{T\geq0}$),
			that is a process similar to a classical Yule model but in which
			the birth processes describing the evolution of the species belonging to each genus start from $m\in \bbN^*$ initial species
			almost surely. To underline the initial condition we will add a superscript $m$ to the Yule process counting the number
			of species for each genus: $\{N_{\lambda}^m(T)\}_{T\geq 0}$.
			We explicitly remark that the letter $m$, used for the initial value of the $m$-Yule model was already used to indicate
			the number of edges from a vertex in the BA model. This choice is not a coincidence, in the next sections we will
			in fact show that, as soon as we create a correspondance between the two models, the initial value $N_{\lambda}^m(0)$, is determined by the parameter $m$ of the BA model.
			Finally, we would like to point out that the $1$-Yule model coincides with the original Yule model of \cite{Yule1925}.

	\section{Main results}\label{mainresults}
		In order to better introduce our main results, let us first describe a heuristic approach explaining the relation between the discrete time process for the degree growth of a fixed vertex and a Yule process.  		
		In the BA model, $m$ directed edges sequentially connect each new vertex to the others
		with probabilities proportional to the degrees of the existing vertices. Thus, at the time at which
		there are $n$ vertices, that is at time $t=n(m+1)$, we have $mn$ directed edges, and 
		by the preferential attachment rule we have approximately
		\begin{align}\label{approx1}
			\bbP[d(v,(n+1)(m+1))=k+1\mid d(v,n(m+1))=k]\approx \frac{km}{2mn}=\frac{k}{2n},
		\end{align}
		where $d(v,t)$ is the degree of $v$ in the BA model. 
		The approximation done in \eqref{approx1} consists in connecting all the $m$ edges simultaneously instead of sequentially,
		that is, we consider $m$ chances of increasing the degree of $v$ from $k$ to $k+1$.
		Furthermore, we neglect the increase of the number of vertices during the random time interval between the instants at which the degree of $v$ changes from $k$ to $k+1$. By formula (\ref{approx1}) the distribution of this random time interval
		is geometric with parameter $k/(2n)$.  In this approximations, when $n\rightarrow\infty$ we obtain a convergence
		to an exponential random variable of parameter $k\lambda$, with $\lambda=1/2$. 
		Moreover, neglecting also the possibility of loops, the initial degree of $v_i$, $i\geq n$, turns out to be equal to $m$.
		These two observations suggest for large values of $n$ to approximate the distribution of the degree of a vertex in the BA model by the distribution of the number of individuals in a Yule process  with parameter $\lambda=1/2$ and initial condition $N_{\lambda}(0)=m$.
		In Theorem \ref{Teo1a}, below, we make rigorous the above heuristics by proving that the process describing the degree of a fixed vertex in the BA model converges in distribution to the number of individuals in a Yule process with initial size $m$.
		Further, much interest is towards the study of the asymptotic degree of a vertex chosen uniformly at random.
		In Theorem \ref{Teo1} we show that the BA model is related
		to a sequence of suitably scaled $m$-Yule models. Exploiting  this relation we prove that  the asymptotic degree distribution of a vertex	chosen uniformly at random in the BA model  coincides with the asymptotic  distribution of the size  of a genus
		chosen uniformly at random in the  $m$-Yule model.	
		\begin{teo}
			\label{Teo1a}
			Let $z(i,w):\bbN^*\times\bbR^+\rightarrow\bbN$, $w\in\bbR^+$ be a function such that
			$c(w):=\lim_{i\rightarrow\infty}z(i,w)/i$ exists finite, with $c(w):\bbR^+\rightarrow \bbR^+$ increasing
			in $w$. Let $b\geq 1$ and $w_1<w_2<\dots<w_b$ be positive real numbers. We have that 
			\begin{align}
				\lim_{i\rightarrow\infty} & \bbP\left[d\left(v_i,(i+z(i,w_1))(m+1)\right)=k_1,\dots,d\left(v_i,(i+z(i,w_b))(m+1)\right)=k_b\right]\\
				&=\bbP[N_{1/2}^m(\log(1+c(w_1)))=k_1,\dots,N_{1/2}^m(\log(1+c(w_b)))=k_b]\nonumber\\           
				&=
				\prod_{\ell=1}^b \binom{k_{\ell}-1}{k_{\ell}-k_{\ell-1}}e^{-\frac{k_{\ell-1}}{2}\log\bigl(\frac{1+c(w_\ell)}{1+c(w_{\ell-1})}\bigr)}
				\Big(1-e^{-\frac{1}{2}\log\bigl(\frac{1+c(w_\ell)}{1+c(w_{\ell-1})}\bigr)}\Big)^{k_{\ell}-k_{\ell-1}}. \notag
			\end{align}
			Here $w_0=0$, $k_0=m$, and $m\le k_1 \le \dots \le k_b\in\bbN^*$.
		\end{teo}
		\begin{obs}
			Notice that, for $\ell=1,\dots,b$, the required time change $i \mapsto i+z(i,w_\ell)$ behaves
			asymptotically as the linear function $i \mapsto i + ic(w_\ell)$.
			Moreover, the logarithm of its slope, $1 + c(w_\ell)$, is the time at which the Yule process is evaluated. Regarding the existence
			of the function $z(i,w_\ell)$, possible choices can be $z(i,w_\ell)=\lfloor iw_\ell\rfloor$ or $z(i,w_\ell)=\lfloor (i-1)w_\ell\rfloor$.
		\end{obs}
		\begin{obs}
			Theorem \ref{Teo1a} states that the joint distribution of the degrees of $v_i$
			at times $z(i, w_1)(m + 1),\dots,z(i, w_b)(m + 1)$ after its first appearance in $(G^t_m)_{t\ge 1}$, converges, as
			$i \to \infty$, to the joint distribution of the number of individuals of a Yule process with $m$ initial
			individuals and parameter $\lambda = 1/2$, evaluated at the times $\log(1 + c(w_1)),\dots, \log(1 + c(w_b))$.
		\end{obs}
		\begin{teo}
			\label{Teo1}
			Consider an $m$-Yule model  $\{Y^m_{1/2,1}(T)\}_{T\geq0}$, and let
			$\mathcal{N}_T^m$ be the size of a genus chosen uniformly at random at time $T$ in  $\{Y^m_{1/2,1}(T)\}_{T\geq0}$.
			Consider the random graph process $({G}_m^t)_{t\geq 1}$  defining the BA model with $N_{k,t}$  vertices with degree $k$.
			Let  $d(V_t)$ be the degree of a vertex chosen uniformly at random at time $t$ in $({G}_m^t)$.
			Then, for $t=n(m+1)$ we have
			\begin{align}
				\label{uniformchoice}
				p_k:=\lim_{n\rightarrow\infty}\bbP(d(V_t)=k)=
				\lim_{T\rightarrow\infty}\bbP(\mathcal{N}_T^m=k), \qquad k \ge m,
			\end{align}
			and for $C>m\sqrt{8}$, 
			\begin{align}
				\label{Inprobability}
				\mathbb{P}\Big(\max_{k}\Big|\frac{N_{k,t}}{n}-\bbP(d(V_t)=k)\Big|\geq C\sqrt\frac{(m+1)\log (n(m+1))}{n}\Big)=o(1).
			\end{align} 
			Furthermore, as $n\rightarrow\infty$, $N_{k,t}/n\rightarrow p_k$ in probability.
		\end{teo}
		Using the previous theorem and  directly exploiting the properties of the $m$-Yule model we are able to recover
		the well-known result for the asymptotic degree distribution of the BA random graph.	
		\begin{prop}
			\label{Teo2}
			Consider an $m$-Yule model  $\{Y^m_{1/2,1}(T)\}_{T\geq0}$  and   the size $\mathcal{N}_T^m$ of a genus chosen uniformly at random
			at time $T$ from  it as in Theorem \ref{Teo1}.
			Then,
	        \begin{equation}
	        	\label{Yule}
				p_k=m(m+1)B(k,3), \qquad k\geq m,
			\end{equation}
			where $B(a,b)$ is the Beta function.
		\end{prop}
		\begin{obs}
			Notice that the distribution \eqref{Yule} coincides with the degree distribution of the BA model \cite{Bollobas2001}.
		\end{obs}
		\begin{obs}\label{cere}
	       	In Section \ref{proofs} we prove the technical Lemma \ref{lema1} on the behaviour of the degree process for
	       	a fixed vertex. Theorems \ref{Teo1a}, \ref{Teo1} and Proposition \ref{Teo2} can also be proved for any random graph process for which such Lemma \ref{lema1} holds. Notice that if necessary, Lemma \ref{lema1} can be extended to the case in which the constant $b_2$ can be taken equal to zero.
	    \end{obs}
	    \begin{obs}\label{independentModel}
	       	An alternative example in which Lemma \ref{lema1} still holds is the ``independent'' model:
	       	for each newly added vertex its $m$  edges are connected to old vertices independently one another. Formally, in Definition  \ref{BAdefinition} replace (\ref{BA}) by
	       	\begin{align}
				\label{BAindependent}
				\mathbb{P}(v_{n+1} \longrightarrow v) =
				\begin{cases}
					\dfrac{d(v,n(m+1))}{2(mn+1)-1}, &  v\neq v_{n+1}, \\ \\
					\dfrac{d(v,n(m+1))+1}{2(mn+1)-1}, & v=v_{n+1}.
				\end{cases}
			\end{align}
			In Section \ref{proofs}, which is devoted to the proofs, Remark \ref{RemarkLemma51} explains why Lemma \ref{lema1} holds for the independent model.
		\end{obs}

		\section{Sampling a random vertex}\label{monotone}

        Before going through the proofs of Theorem \ref{Teo1a} and Theorem \ref{Teo1},  we introduce here the general notion of \emph{planted model} and a fundamental procedure we will make use in the next section to prove the relationship
		between two random quantities in the BA model and in the $m$-Yule model. In particular, we will put in relation the degree of a vertex chosen uniformly at random in the BA model and the number of species of a genus
		chosen uniformly at random from one of the $m$-Yule models,
		also chosen uniformly at random from the set of all $m$-Yule models in the planted model.
		
		For a greater generality we
		consider the case in which the number of edges added each time a vertex appears, form a sequence $\{M_j\}_{j\geq1}$ of random variables taking values in $\bbN^*$ almost surely.
		This result can be easily specialized to the case of the BA model,
		that is $M_j=m$ a.s.\ for every $j$.
		An example is the random graph related to Simon model (see \cite{PachonPolitoSacerdote}): it can be related to a Yule model where the above random variables are independent and geometrically distributed.
		
		In order to describe the sampling procedure, we introduce first a model that we call the \emph{planted model} for the random graph
		$(\bbG_t)_{t\geq1}$. The idea underlying the planted model is to superimpose a tree structure on the graph
		which is independent of the degree processes. 

		We start by noting that, at each time of the form
		$\mathfrak{T}_i=\sum_{r=1}^i(M_r+1)$, the graph $\bbG_{\mathfrak{T}_i}$ has exactly $i$ vertices, $i\in\bbN^*$. We refer to them
		as the \emph{planted vertices}.
		Let us now consider the value $i$ to be fixed; to obtain the tree structure at the following times $\mathfrak{T}_{n+1}$, $n\geq i$,
		we attribute to $v_{n+1}$ the role of child of a vertex chosen uniformly at random
		from the set of the existing vertices $\{v_1,v_2,\dots,v_n\}$.
		Iterating this procedure we obtain chains of successive offsprings of each of the planted vertices $\{v_1,v_2,\dots,v_i\}$.
		Further, we call a vertex $v$ that appeared after $v_j$, $j=1,\dots,i$, a descendant of $v_j$ if both $v$ and $v_j$ belong to the same ancestral line.
		We order the descendant of $v_j$ by renaming $v$ as $v_{j,\ell}$, if $v$ is the $\ell$-th descendant of $v_j$ and
		denote $v_j$ as $v_{j,1}$, that is, $v_j$ is in turn, its first descendant.
		In this way we construct $i$ birth processes in discrete time, $\{b(v_j,\mathfrak{T}_n)\}_{n \ge i}$, $j=1,\dots,i$.
		Here $b(v_j,\mathfrak{T}_n)$, $j=1,\dots,i$, $n \ge i$, is the total number of descendants of $v_j$ at time $\mathfrak{T}_n$.
		Table \ref{tableuno} shows an example of the construction of the planted model.
		Note that we have: 
		\begin{itemize} 
			\item $b(v_j,\mathfrak{T}_i)=1$, $j=1,\dots,i$;
			\item $\bbP[b(v_j,\mathfrak{T}_{n+1})= k+1\mid b(v_j,\mathfrak{T}_n)= k]=k/n$, $k \geq 1$, $n\ge i$, $j=1,\dots,i$.
		\end{itemize}
 
		The second equality holds because at time $\mathfrak{T}_{n+1}$, $n\geq i$, each existing vertex in the set $\{v_1,v_2,\dots,v_n\}$
		may give birth to a new one with probability $1/n$. 
			
		Note that, for a fixed $i \ge 1$, the planted model is defined for $n\geq i$. Thus for example, given a value
		of $i$ there is
		no process $\{b(v_{j},\mathfrak{T}_n)\}_{n\geq i}$ with $j>i$, because $j$ has to be an element of $\{1,\dots,i\}$.
		The dynamic of the planted model then proceeds for $n \geq i$. Finally, note that
		the $i$ discrete-time birth processes are exchangeable.
		
		\begin{table} 
			\centering
			\begin{tabular}{c|c:c|c:c|c|c:c}
				$n$ & & $b(v_1,\mathfrak{T}_n)$ & & $b(v_2,\mathfrak{T}_n)$ & $\dots$ & & $b(v_i,\mathfrak{T}_n)$ \\\hline & & & & & & &\\
				$i$ & $v_1=v_{1,1}$ & 1 & $v_2=v_{2,1}$ & 1 & & $v_i=v_{i,1}$ & 1 \\
				$i+1$ & & 1 & $v_{i+1}=v_{2,2}$ & 2 & & & 1 \\
				$i+2$ & & 1 & & 2 & $\dots$ & $v_{i+2}=v_{i,2}$ & 2 \\
				$i+3$ & $v_{i+3}=v_{1,2}$ & 2 & & 2 & & & 2 \\
				$i+4$ & & 2 & $v_{i+4}=v_{2,3}$ & 3 & & & 2 \\
				$\dots$ & $\dots$ & $\dots$ & $\dots$ & $\dots$ & $\dots$ & $\dots$ & $\dots$
			\end{tabular}
			\caption{\label{tableuno}(First line): The construction starts with $i$ discrete-time birth processes at time
				$\mathfrak{T}_i$, each one with one individual. (Second line): At time 
				$\mathfrak{T}_{i+1}$ a new vertex $v_{i+1}$ appears. The vertex $v_{i+1}$ is assigned as a child to one of
				the existing vertices $\{v_1,v_2,\dots,v_i\}$  with probability $1/i$. In this table the appearing vertex $v_{i+1}$
				becomes a child of $v_2$ and
				consequently it is renamed as $v_{2,2}$, that is the second individual in the birth process relative to $v_2$.  
				(Next lines): At times $\mathfrak{T}_{n+1}$, $n\geq i$, the vertex $v_{n+1}$ appears and it is assigned
				to one of the existing vertices with probability $1/n$. Observe that in this example
				$b(v_2,\mathfrak{T}_{i+4})=3$. Given this information, $\bbP[b(v_2,\mathfrak{T}_{i+5})=4]=3/(i+4)$.} 
		\end{table}

		\subsubsection{Sampling from the planted model}\label{samplingPlanted}
			
			Consider the following procedure. Given a realization of $\{b(v_j,\mathfrak{T}_n)\}_{n \ge i}$, $j=1,\dots,i$,
			\begin{enumerate}
				\item choose one of the $i$ discrete-time birth processes with probability proportional to the number of its vertices;  
				\item choose a vertex uniformly at random among those belonging to the realization of the selected birth process.
			\end{enumerate}
			Our  focus is on the selected vertex $v_{j,\ell}$,  $j=1,\dots, i$, $ \ell=1,\dots, b(v_j,\mathfrak{T}_n)$
			and on the selected birth process. Let $W$ be the index of the birth process chosen. Plainly, $W$ takes values in	$\{1,\dots,i\}$ almost surely.
			\begin{prop}
				\label{plantedinduce}
				It holds,
				\begin{enumerate}
					 \item $\bbP(W=j)=1/i$, $j \in \{1,\dots,i\}$,
					 \item $\bbP(\{v_{j,\ell} \text{ is selected}\})=1/n$.
				\end{enumerate}
			\end{prop}
			\begin{proof}
				It immediately follows from the exchangeability of the $i$ discrete-time birth processes.
			\end{proof}
			\begin{obs}
				The suggested algorithm is a way to select a vertex uniformly at random from $\bbG_{\mathfrak{T}_n}$, $n\geq i$,
				and refers to a given realization of the $i$ birth processes $\{b(v_j,\mathfrak{T}_n)\}_{n \ge i}$, $j=1,\dots,i$.
				Averaging on all possible realizations of the $i$ birth processes we actually select a vertex uniformly at random:
				we first choose one of the $i$ birth processes belonging to the planted model with uniform probability, then
				we select a vertex among those belonging to the chosen birth process again with uniform probability.
			\end{obs}
	
	\section{Proofs}\label{proofs}
	
		We first give a brief outline of the proofs of the main results described in Section \ref{mainresults}.
		Regarding Theorem \ref{Teo1a},  we  start by showing that the transition probabilities of the degree process of a fixed vertex with sufficiently large index in the
		BA model are bounded above and below (Lemma \ref{lema1}). With these bounds we construct two Markov processes coupled
		with the original degree process 
		(Lemma \ref{teo2.1} and Corollary \ref{cor2}).
		In Lemma \ref{lema1conv}, we exploit this coupling to conclude that the finite-dimensional distribution
		of the degree of a vertex in the BA model converges to the finite-dimensional distribution of the number
		of individuals in a Yule process with initial
		population size equal to $m$. Notice that this result is consistent with the analysis of preferential attachment trees
		performed through continuous-time branching processes (see e.g.\ \cite{RSA:RSA20137,Bhamidi}).
			
		To prove Theorem \ref{Teo1} we proceed according to the following steps.
		First, by making use of the planted model and the sample procedure from the planted model described in Section \ref{monotone}, we make explicit the relationship between the deterministic appearance of new vertices in the BA model and the random appearance of new births in a continuous-time Yule process.
		The key point is that, by Theorem \ref{plantedinduce}, the choice of a vertex with uniform distribution in the BA model
		is equivalent to choosing a birth process from the planted model with uniform distribution and then
		choosing a vertex among those belonging to the selected birth process, again with uniform distribution. 
		Then, in Lemma \ref{S-Yule2} we prove  that the number of individuals in each birth process
		of the planted model, say $\{b^j_i\}_{i \ge 1}$, where $b^j_i=\{b(v_j,n(m+1))\}_{n\geq i}$, $1\leq j\leq i$,  converges in distribution as $i\rightarrow\infty$,
		to the size of a Yule process with parameter $\beta=1$ and with one initial progenitor.

		\subsection{Auxiliary lemmas and the proof of Theorem \ref{Teo1a}}
		
			\label{convergencia}
			The proof of Theorem \ref{Teo1a} can be summarized in two main steps. Within the structure of the BA model we first identify two different counting processes
			in discrete time,				
			one for the appearing of in-links of each specific vertex  and the other related to the creation of new vertices.
			Then, we  prove that these two processes converge to the two birth processes which are at the basis
			of the definition of an $m$-Yule model.
			
			Before  starting the construction of the process for the appearance of in-links
			of a fixed vertex we introduce the following definition. 
			\begin{defn}
				\label{complete}
				We say that a vertex $v_i$  appears  ``complete'' when   it has appeared in the BA random graph
				process together with all the directed edges  originated from it.
			\end{defn}
			Note that the degree of a complete
			vertex is at least $m$, and  at time $t=n(m+1)$,  the BA  model has for the first time exactly $n$ complete vertices.

			Next we determine how the degree of a fixed vertex $v_i$, for a sufficiently large $i$, changes during the time until a new complete vertex appears. 
			\begin{lem}
				\label{lema1}
				Let  $({G}_m^t)_{t\geq 1}$ be the random graph process defining the BA model
				and let $d(v_i,t)$ denote  the degree of an existing vertex $v_i$ at time $t$, $i\leq n$.
				Given that $d(v_i,n(m+1))=k$, $n > k\geq m$, for sufficiently large $i$ there exist constants $b_1>b_2>0$ and $c_1,c_2>0$ such that
				\begin{align}
					\label{transitiond}
					\frac{k}{2(n+1)}+c_2\left(\frac{k}{n}\right)^2& <\bbP[d(v_i,(n+1)(m+1))=k+1 | d(v_i,n(m+1))=k] \notag\\
					& < \frac{k}{2n}+c_1\left(\frac{k}{n}\right)^2,
				\end{align}
				and, for $m > 1$,
				\begin{align}
					\label{transitiondn}
					b_2\left(\frac{k}{n}\right)^2\leq\bbP[k+2 \le d(v_i,(n+1)(m+1)) \le k+m| d(v_i,n(m+1))=k]\leq b_1\left(\frac{k}{n}\right)^2.
				\end{align}
				Furthermore,
				\begin{align}
					\label{initCond}
					\bbP[d(v_{n+1},(n+1)(m+1))=m]=\prod_{\ell=2}^{m+1}\left(1-\frac{1}{2(mn+\ell-1)-1}\right)=1-\Theta(1/n),
				\end{align}
				where we make use of the asymptotic Big Theta notation \cite{knuth}.
			\end{lem}

			\begin{proof}
				Our aim is to determine the change of degree of a fixed vertex during the time interval $(n(m+1),(n+1)(m+1)]$,
				i.e., during the time interval necessary to switch from $n$ to $(n+1)$ complete vertices.
				
				Let us fix $t=n(m+1)$ and follow the graph growth during the considered interval.
				At time $n(m+1)+1$ a new vertex $v_{n+1}$ (without edges) appears.
				Then from time  $n(m+1)+2$ to time $(n+1)(m+1)$ a directed edge from $v_{n+1}$ to an existing vertex $v_i$, $i\leq n+1$,
				is added. The vertex  $v_i$ is chosen with probability given by (\ref{BA}). Let  $Y_{v_i}^{n}$ be the total number
				of incoming edges  to $v_i$,  $i\leq n$, added to $v_i$ during the time interval $(n(m+1),(n+1)(m+1)]$.				
				Note that $\bbP[d(v_i,(n+1)(m+1))=k+\ell\mid d(v_i,n(m+1))=k]=\bbP[Y_{v_{i}}^{n}=\ell\mid d(v_i,n(m+1))=k]$, $\ell=0,\dots, m$.
				To estimate the latter conditional probabilities for a sufficiently large $i$ we
				distinguish the cases $Y_{v_{i}}^{n}=0$, $Y_{v_{i}}^{n}=1$, and $Y_{v_{i}}^{n}\geq2$.

				In the first case, considering the probabilities (\ref{BA}) we have
				\begin{align}
					\label{d0}
					\bbP[Y_{v_{i}}^{n}=0\mid d(v_i,n(m+1))=k]=\prod_{\ell=2}^{m+1}\left(1-\frac{k}{2(mn+\ell-1)-1}\right).
				\end{align}
				Since $\prod_{\ell=2}^{m+1}\left(1-\frac{k}{2(mn+\ell-1)-1}\right) \leq \left(1-\frac{k}{2(mn+m)-1}\right)^m$,
				we get the upper bound for (\ref{d0}),
				\begin{align}
					\label{ceromenor}
					\left(1-\frac{k}{2(mn+m)-1}\right)^m=1-\frac{mk}{2m(n+1)-1}+ O\left(\frac{k^2}{n^2}\right).
				\end{align}
				Furthermore, since
				$\prod_{\ell=2}^{m+1}\left(1-\frac{k}{2(mn+\ell-1)-1}\right) \geq \left(1-\frac{k}{2(mn+1)-1}\right)^m$, we get the lower bound
				\begin{align}
					\label{ceromayor}
					\left(1-\frac{k}{2(mn+1)-1}\right)^m=1-\frac{mk}{2mn+1}+ O\left(\frac{k^2}{n^2}\right).
				\end{align}

				Now we move first to the third case. We observe that if $m=1$
				then $\bbP(Y_{v_{i}}^{n}\geq2\mid d(v_i,n(m+1))=k)=0$.	Thus, we calculate such probability for $m>1$ only.  
				Furthermore, since we do not need a closed form of $\bbP(Y_{v_{i}}^{n}\geq2\mid d(v_i,n(m+1))=k)$,
				we limit ourselves to estimate  its  order of magnitude. 
				For each $y=2,\dots,m$, the event $\{Y_{v_{i}}^{n}=y\}$ means that 
				$v_i$ gets $y$ new incoming edges joining $v_i$ at the times $t=n(m+1)+\ell$, $\ell=2,\dots,m+1$.
				Given the value of the degree of $v_i$ at time $t-1$, considering (\ref{BA}), a new edge is attached to $v_i$ at time
				$t=n(m+1)+\ell$ with probability
				\begin{align}\label{pvnl}
					p_{v_i}^{n,\ell}:=\frac{d(v_i,n(m+1)+\ell-1)}{2(mn+\ell-1)-1}.
				\end{align}
				Let $\Omega$ be the space of all sequences of $m$ dichotomous independent experiments, performed at times
				$t=n(m+1)+\ell$, $\ell=2,\dots,m+1$, with exactly $y$ successes. Assume that $p_{v_i}^{n,\ell}$, $\ell=2,\dots,m+1$, are the
				probabilities of success. Note that the cardinality of $\Omega$ is equal to that of the set of all $y$-combinations from a
				given set of $m$ distinct elements, i.e.\ $|\Omega|=\binom{m}{y}$. Take the set $\{2,\dots,m+1\}$ and consider its $y$-combinations,
				say $C_y=\{e_1,\dots,e_{\binom{m}{y}}\}$ (e.g.\ ordered by their smallest element). For each $e\in C_y$, let $e(j)$ denote
				the position of the  $j$-th success in $e$, $j=1,\dots,y$. We have,			
				\begin{align}\label{newdos}
					\bbP(Y_{v_{i}}^{n}=y| d(v_i,n(m+1))=k)&=\sum_{e\in C_y}\prod_{j=1}^y p_{v_i}^{n,e(j)}
					\prod_{\ell\in\{2,\dots,m+1\},\ell\notin e(1),\dots,e(y)}(1-p_{v_i}^{n,\ell})\nonumber\\
			   	&=\binom{m}{y}\Theta\left(\frac{k^y}{n^y}\right)\left[1-\Theta\left(\frac{k}{n}\right)\right]^{m-y}\nonumber\\
				&=\binom{m}{y}\Theta\left(\frac{k^y}{n^y}\right)\sum_{\ell=0}^{m-y}
				\binom{m-y}{\ell}(-1)^{\ell}\Theta\left(\frac{k^\ell}{n^{\ell}}\right)\nonumber\\
					&=\Theta\left(\frac{k^y}{n^y}\right), \qquad 2\leq y\leq m. 
				\end{align}
				Hence,
				\begin{align}
					\label{dos}
					\bbP(Y_{v_{i}}^{n}\geq2\mid d(v_i,n(m+1))=k)
					=\Theta\left(\frac{k^2}{n^2}\right).
				\end{align}
				Finally, by  (\ref{ceromayor}) and (\ref{dos}) we obtain that
				$\bbP(Y_{v_{i}}^{n}=1\mid d(v_i,n(m+1))=k)$ is at most
				\begin{align}
					\label{unomenor}
				1-\left[1-\frac{mk}{2mn+1}+O\left(\frac{k^2}{n^2}\right)\right]-\Theta\left(\frac{k^2}{n^2}\right)
					<\frac{k}{2n}+O\left(\frac{k^2}{n^2}\right),
				\end{align}
				and by (\ref{ceromenor}) and (\ref{dos}), $\bbP(Y_{v_{i}}^{n}=1\mid d(v_i,n(m+1))=k)$ is at least
				\begin{align}
					\label{unomayor}
				1-\left[1-\frac{mk}{2m(n+1)-1}+O\left(\frac{k^2}{n^2}\right)\right]-\Theta\left(\frac{k^2}{n^2}\right)
					>\frac{k}{2(n+1)}+O\left(\frac{k^2}{n^2}\right).
				\end{align}
				Therefore, for sufficiently large $i$ we have that, for each $k\ge m$, there exist $b_1>b_2>0$ such that (\ref{dos})
				gives (\ref{transitiondn}), and $c_1,c_2>0$ such that  (\ref{unomenor}) and (\ref{unomayor}) give (\ref{transitiond}).

				In order to determine (\ref{initCond}),  let $X_{v_{n+1}}^{n}$ be the number of incoming edges  from $v_{n+1}$ to itself
				during the time interval $(n(m+1), (n+1)(m+1)]$,
				that is the number of loops. Note that during this period  $X_{v_{n+1}}^{n}$ can be at most
				equal to $m$, since at time $n(m+1)+1$ no edge is added.   
				Thus, by (\ref{BA}), the probability of no loops for  $v_{n+1}$  during such time interval  is given by 
				\begin{align}\label{eventosEquiv}
					\bbP(X_{v_{n+1}}^{n}=0)=\prod_{i=2}^{m+1}\left(1-\frac{1}{2(mn+i-1)-1}\right)= [1-\Theta(1/n)]^m= 1-\Theta(1/n).
				\end{align}
				If the number of loops for $v_{n+1}$ is zero, this is equivalent to say that when $v_{n+1}$ appears
				complete, its degree is equal to $m$. Thus, by (\ref{eventosEquiv}) we can write
				$\bbP[d(v_{n+1},(n+1)(m+1))=m]=1-\Theta(1/n)$, so that the proof is complete.
			\end{proof}

			\begin{obs}\label{RemarkLemma51}
				For the independent model described in Remark \ref{independentModel}, note that the left-hand side of formula (\ref{d0}) is equal to (\ref{ceromayor}), formula (\ref{pvnl}) is equal to $p_{v_i}^{n,2}$ for all $\ell=2,\dots,m+1$, $\bbP(Y_{v_{i}}^{n}=1\mid d(v_i,n(m+1))=k)$ is equal to (\ref{unomenor}) and 
				\begin{equation}
					\mathbb{P}[d(v_{n+1},(n+1)(m+1))=m]=\left(1-\frac{1}{2(mn+1)-1}\right)^m=1-\Theta(1/n).
				\end{equation}
			\end{obs}

			Now we consider the degree process	$\{d(v_i,n(m+1))\}_{n\geq i}$,
			indexed by $n$, where  $d(v_i,n(m+1))$
			satisfies (\ref{transitiond}), (\ref{transitiondn}) and (\ref{initCond}). Let $E=\{m,m+1,\dots\}$
			be the state space of the process $\{d(v_i,n(m+1))\}_{n \ge i}$ and let $\bbM(E)$ be the class
			of probability measures on the space $E$ endowed
			with the $\sigma$-algebra $\bbF=\mathcal{P}(E)$, the power set  of $E$. The degree process	$\{d(v_i,n(m+1))\}_{n\geq i}$ is defined
			on the product space $(E^{\infty},\bbF^{\infty})=(\times_{n=i}^{\infty}E,\otimes_{n=i}^{\infty}\bbF)$.
			The process from time $i$ to time $i+h$, $\{d(v_i,n(m+1))\}_{n=i}^{i+h}$ takes values in the
			product space $(E^{h},\bbF^{h})=(\times_{n=i}^{i+h}E,\otimes_{n=i}^{i+h}\bbF)$.
			The elements of the spaces $(E^{h},\bbF^{h})$ and $(E^{\infty},\bbF^{\infty})$
			will be denoted by $x^{i+h}=(x_i,x_{i+1},\dots,x_{i+h})$ and $x^{\infty}=(x_i,x_{i+1},\dots)$, respectively.
			We say that $x^{i+h}\leq y^{i+h}$ if and only if $x_{i+j}\leq y_{i+j}$ for all $0\leq j\leq h$.
			
			To prove Theorem \ref{Teo1a}, we proceed through two steps:
			
			\begin{enumerate}
				\item we define two Markov processes,
					on the same probability space as  $\{d(v_i,n(m+1))\}_{n\geq i}$, determined by suitable Markov kernels and
					we show that those two processes bound  from above and below the degree process of the BA model (Lemma \ref{teo2.1} and Corollary \ref{cor2});
				\item we prove that those two processes, each of them evaluated at a convenient time, converge in
					distribution as $i\rightarrow\infty$ (and therefore as $n \rightarrow \infty$) to a unique process evaluated at a unique time (Lemma \ref{lema1conv}).
			\end{enumerate}
			
			With respect to the first step, for any $m>1$, we define $p_{n+1}$ to be a positive function  on $E^n\times E$,
			measurable with respect to $\bbF^n\otimes \bbF$,
			and given by
			\begin{align}\label{pijk}
				p_{n+1}(x^n,x)=p_{n+1}(x^n,x_n+\ell)=
				\begin{cases}
					\frac{x_n}{2(n+1)}+c_2\left(\frac{x_n}{n}\right)^2, & \quad \ell=1, \\
					b_2\left(\frac{x_n}{n}\right)^2, & \quad \ell=2,\\
					1-\frac{x_n}{2(n+1)}-(b_2+c_2)\left(\frac{x_n}{n}\right)^2, & \quad  \ell=0,\\
					0, & \text{otherwise.}
				\end{cases}
			\end{align}
			Note that this function depends only on $x_n$, the last element of $x^n$, and $n$.
			Then we define the following Markov transition kernel  $K_{n+1}^p$ from  $E^n\times\bbF$ into $[0,1]$: 
			\begin{align}
				K_{n+1}^p(x^n,B)=\sum_{x\in B}p_{n+1}(x^n,x),\qquad x^n\in E^n, \: B\in\bbF.
			\end{align}
			The mapping $B \rightarrow K_{n+1}^p(x^n, B)$  is a measure $P_{n+1}\in\bbM(E)$ for every $x^{n}\in E^{n}$.
			Similarly we define the function 
			\begin{align}
				\label{rijk}
				r_{n+1}(z^n,z)=r_{n+1}(z^n,z_n+\ell)=
				\begin{cases}
					\frac{z_n}{2n}+c_1\left(\frac{z_n}{n}\right)^2, & \quad \ell=1, \\
					b_1\left(\frac{z_n}{n}\right)^2, & \quad \ell=m,\\
					1-\frac{z_n}{2n}-\left( b_1+c_1 \right)\left(\frac{z_n}{n}\right)^2, & \quad  \ell=0,\\
					0, & \text{otherwise,}
				\end{cases}
			\end{align}
			that we associate to the Markov transition kernel $K_{n+1}^r$,
			where $B \rightarrow K_{n+1}^r(z^n, B)$, is a measure $R_{n+1}\in\bbM(E)$ for every $z^{n}\in E^{n}$.
			Note that if $x_i\leq z_i$, $b_2<b_1$ and letting $c_1>c_2$, then by (\ref{pijk}) and (\ref{rijk}),
			$x_n\leq z_n$, $n> i$. Take  $y_n\in E$ such that, $x_n\leq y_n\leq z_n$, $n\geq i$. Then,
			from (\ref{pijk}) and (\ref{rijk}), there exists a
			function $q_{n+1}(y^n,y) = q_{n+1}(y^n,y_n+\ell)$, such that:
			\begin{align}
				\label{qijk}
				&p_{n+1}(x^n,x_n+1) < q_{n+1}(y^n,y_{n}+1)< r_{n+1}(z^n,z_n+1),\\
				&p_{n+1}(x^n,x_n+2) < \sum_{\ell=2}^m q_{n+1}(y^n,y_{n}+\ell)< r_{n+1}(z^n,z_n+m), \nonumber\\
				&q_{n+1}(y^n,y_{n}) = 1-\sum_{\ell=1}^m q_{n+1}(y^n,y_{n}+\ell), \notag
			\end{align}
			whenever $x_n\leq y_n\leq z_n$, $n\geq i$.
			We associate this function to a further Markov transition kernel $K_{n+1}^q$ in the same way as $K_{n+1}^p$
			and  $K_{n+1}^r$, where $B \rightarrow K_{n+1}^q(y^n, B)$  is a measure $Q_{n+1}\in\bbM(E)$
			for every $y^{n}\in E^{n}$. 
			
			In order to prove that there exist two processes bounding respectively from above and below the degree process of
			the BA model we first need the following result:
			\begin{lem}
				\label{teo2.1}
				Let  $X_i$, $Y_i$, and $Z_i$, be random variables on $E$ with distributions $P_i$, $Q_i$ and $R_i$,
				respectively, and satisfying $\bbP(X_i=Y_i=Z_i)=1$.  Then there exist random variables $X_{n+1}$, $Y_{n+1}$,
				and $Z_{n+1}$, $n\geq i$, taking values in $E$,
				such that  the conditional distributions of $X_{n+1}$ given $X_n=x_n$,
				$Y_{n+1}$ given $Y_n=y_n$, and $Z_{n+1}$ given $Z_n=z_n$, are exactly
				$p_{n+1}(x^n,\cdot)$, $q_{n+1}(y^n,\cdot)$, and $r_{n+1}(z^n,\cdot)$,
				respectively. Moreover,
				\begin{align}
					\label{coupling1}
					\bbP(X_n\leq Y_n\leq Z_n, \: n=i, \: i+1,\: \dots)=1.
				\end{align}
			\end{lem}

			\begin{proof}
				We seek to prove a stochastic ordering for $K_{n+1}^p(x^{n},\cdot)$, $K_{n+1}^q(y^{n},\cdot)$ and $K_{n+1}^r(z^{n},\cdot)$.
				For this aim, take the set $B\in \bbF$ such that $B:=\{b,b+1,\dots\}$, where $b$ is any integer $b\geq m$. Then,
				\begin{align}\label{kernelp}
					K_{n+1}^p(x^{n},B)=\sum_{j\geq b}p_{n+1}(x^n,j)=
					\begin{cases}
						1,& \quad b\leq x_n, \\
						\frac{x_n}{2(n+1)}+(c_2+b_2)\left(\frac{x_n}{n}\right)^2, & \quad b=x_n+1, \\
						b_2\left(\frac{x_n}{n}\right)^2, & \quad b=x_n+2,\\
						0, & \quad b\geq x_n+3,
					\end{cases}
				\end{align}
				\begin{align}\label{kernelq}
					K_{n+1}^q(y^{n},B)=\sum_{j\geq b}q_{n+1}(y^n,j)=
					\begin{cases}
						1,&  \quad b\leq y_n, \\
						\sum_{i=\ell}^m q_{n+1}(y^n,y_n+\ell), & \quad b=y_n+\ell, \\ & \qquad \ell=1,\dots,m-1, \\
						q_{n+1}(y^n,y_n+m),& \quad b=y_n+m,\\
						0, & \quad b\geq y_n+m+1,
					\end{cases}
				\end{align}
				and
				\begin{align}\label{kernelr}
					K_{n+1}^r(z^{n},B)=\sum_{j\geq b}r_{n+1}(z^n,j)=
					\begin{cases}
						1,& \quad b\leq z_n, \\
						\frac{z_n}{2n}+(c_1+b_1)\left(\frac{z_n}{n}\right)^2, & \quad b=z_n+1, \\
						b_1\left(\frac{z_n}{n}\right)^2, & \quad b=z_n+2,\\
						0, & \quad b\geq z_n+3.
					\end{cases}
				\end{align}			 
				Since $X_i=Y_i=Z_i$ a.s., $b_2<b_1$, and $c_2<c_1$, then by (\ref{pijk}), (\ref{rijk}) and (\ref{qijk}),
				we obtain that $x_n\leq y_n\leq z_n$, $n\geq i$.
				Thus comparing the three kernels (\ref{kernelp}), (\ref{kernelq}) and (\ref{kernelr}),
				we have
				\begin{align*}
					K_{n+1}^p(x^n,B)\leq K_{n+1}^q(y^n,B)\leq K_{n+1}^r(z^n,B). 
				\end{align*}
				Equivalently, $K_{n+1}^p(x^n,\cdot)$ is stochastically smaller than $K_{n+1}^q(y^n,\cdot)$, and the latter is in turn
				stochastically smaller than $K_{n+1}^r(z^n,\cdot)$, whenever $x^n\leq y^n\leq z^n$. 
				To show that (\ref{coupling1}) holds we finally apply Theorem 2 in \cite{StIneq}.
			\end{proof}

			Let us consider the process $\{d(v_i,n(m+1))\}_{n\geq i}$ and its probability space
			$(\Omega,\mathcal{A}, \bbP)$. On the same probability space let us define two Markov processes  
			$\{d^1(v_i,n(m+1))\}_{n\geq i}$ and $\{d^2(v_i,n(m+1))\}_{n\geq i}$, with their
			initial states such that $\bbP\big(d^1(v_i,i(m+1))=d(v_i,i(m+1))=d^2(v_i,i(m+1))\big)=1$, and
			transition probabilities given by
			(\ref{rijk}) and (\ref{pijk}), respectively.
	
			\begin{cor}
				\label{cor2}
				For $i$ sufficiently large, there exist versions   $\{\tilde d^1(v_i,n(m+1))\}_{n\geq i},\{\tilde d^2(v_i,n(m+1))\}_{n\geq i}$,
				$\{\tilde d(v_i,n(m+1))\}_{n\geq i}$ of the processes $\{d^1(v_i,n(m+1))\}_{n\geq i}$,
				$\{d^2(v_i,n(m+1))\}_{n\geq i}$ and $\{d(v_i,n(m+1))\}_{n\geq i}$, respectively, such that 
				\begin{align}
					\label{coupling}
					\bbP[\tilde d^2(v_i,n(m+1))\leq  \tilde d(v_i,n(m+1))\leq \tilde d^1(v_i,n(m+1)), n=i,i+1,\dots]=1. 
				\end{align} 
			\end{cor}
			
			\begin{proof}
				It immediately follows by applying Lemma \ref{teo2.1} to
				$\{d^1(v_i,n(m+1))\}_{n\geq i}$, $\{d^2(v_i,n(m+1))\}_{n\geq i}$, and $\{d(v_i,n(m+1))\}_{n\geq i}$. 
			\end{proof}
			\begin{lem}
				\label{lema1conv}
				Let $\{\tilde d(v_i,n(m+1))\}_{n\geq i}$, be the  process of Corollary \ref{cor2}, $w\in\bbR^+$ and let
				$z(i,w):\bbN^*\times\bbR^+\rightarrow\bbN^*$ be a function such that $c(w):=\lim_{i\rightarrow\infty}z(i,w)/i$ exists finite,
				where $c(w):\bbR^+\rightarrow \bbR^+$ is an increasing function in $w$. Let $b\geq1$ and $w_1<w_2<\dots<w_{b}$ be positive
				real numbers.				
				Then, the random vector
				\begin{align*}
					\left(\tilde d(v_i,(i+z(i,w_1))(m+1)),\dots,\tilde d(v_i(i+z(i,w_b))(m+1))\right)
				\end{align*}
				converges in distribution to $\bigl(N_{1/2}^m(\log(1+c(w_1))),\dots,N_{1/2}^m(\log(1+c(w_b)))\bigr)$
				as $i\rightarrow\infty$. Here $N_{1/2}^m(T)$, $T\geq0$, is the number of individuals
				at time $T$ in a Yule process with parameter $1/2$ and $m$ initial individuals. 
			\end{lem}

			\begin{proof}
				In order to prove the convergence, we make use of the processes  $\{\tilde d^1(v_i,n(m+1))\}_{n\geq i}$
				and $\{\tilde d^2(v_i,n(m+1))\}_{n\geq i}$ and of their behaviour as $i$ goes to infinity.
				We focus now only on the process  $\{\tilde d^1(v_i,n(m+1))\}_{n\geq i}$ as the case of $\{\tilde d^2(v_i,n(m+1))\}_{n\geq i}$
				can be treated analogously.

				Let $i$ be fixed and let  $T_i^0=0$. 
				For every $x\geq 1$ we introduce the times $T_i^x=\sum_{n=i}^{i+x-1}1/n$. In this way we obtain a partition of $(0,T_i^x]$, 
				\begin{align}\label{partition}
					(0,T_i^x]=(0,T_i^1]\cup(T_i^1,T_i^2]\cup\dots\cup(T_i^{x-1},T_i^{x}].
				\end{align}
				The intervals of this partition have lengths $h_n=1/n$, $n=i,i+1,\dots,i+x-1$
				(see Figure \ref{ScalTime}).
				\begin{figure}
					\center
					\includegraphics[scale=0.3]{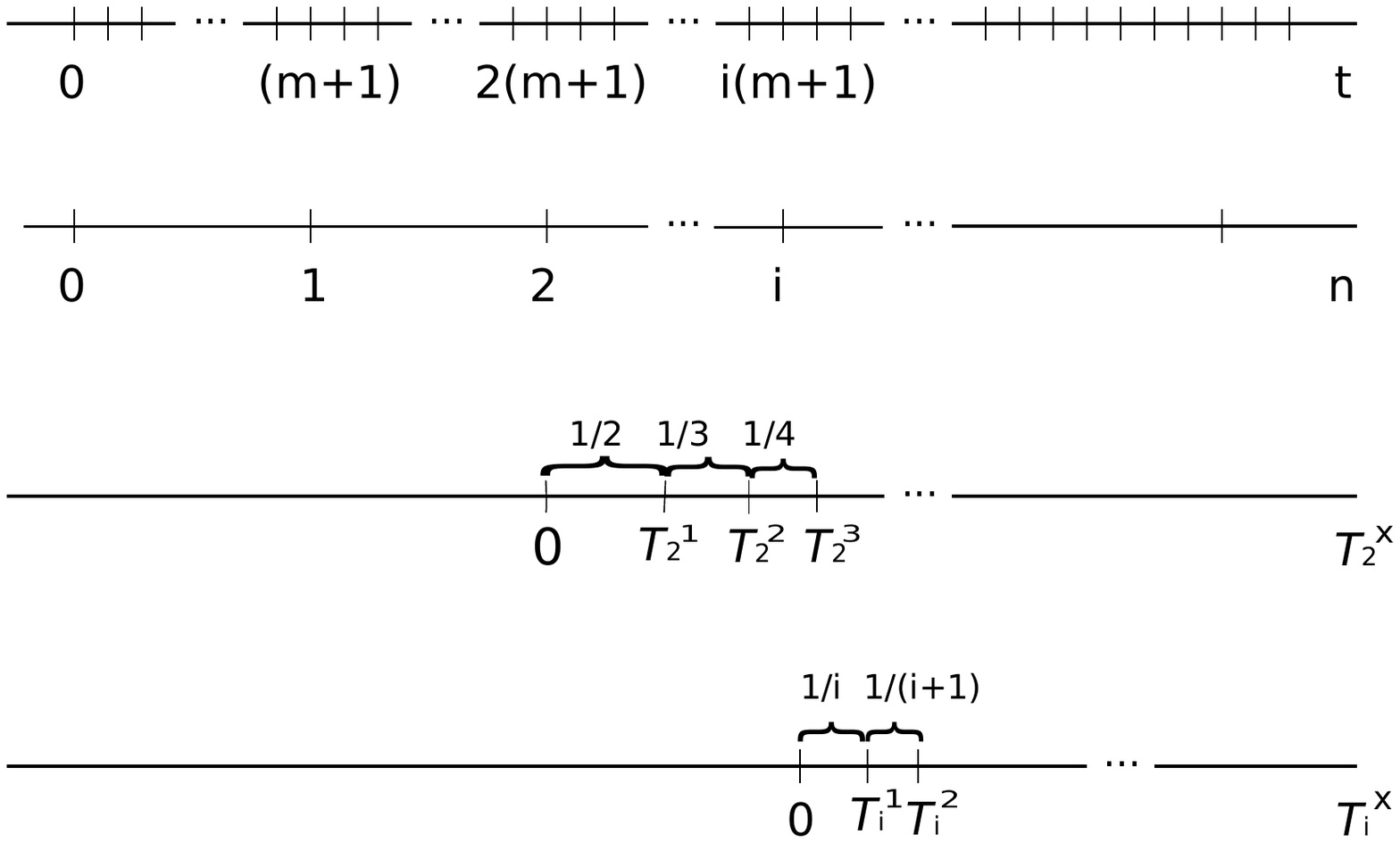}
					\caption{\label{ScalTime}The first line represents the time axis of the BA model. The second line shows the number of
						complete vertices in the BA model. The third and fourth lines correspond to the partitions
						of $(0,T_2^x]$ and $(0,T_i^x]$, respectively.}
				\end{figure}

				We introduce the point process $\{\textswab{N}^{1,i}(T)\}_{T\geq0}$,
				jumping at times $T_i^x$, $x\geq1$, and determined by the following rules:
				\begin{enumerate}
					\item At time $T=0$ the process starts with an initial random number
						of individuals supported on $\{m,m+1,\dots,2m\}$ and with distribution only asymptotically degenerate on $m$, i.e. 
						\begin{equation}\label{initialcond1}
							\bbP(\textswab{N}^{1,i}(0) \ne m)=1-\prod_{\ell=2}^{m+1}\Big(1-\frac{1}{2(mi+\ell-1)}\Big)=O(1/i).
						\end{equation} 
					\item The transition probabilities of this  point process  coincide with (\ref{rijk}) when $n=i+x-1$ and $z_n=k$, for every fixed $k \in \mathbb{N}^*$, $k \ge m$, and $x\geq 1$. We write these probabilities here using asymptotic notation. For each $x\geq 1$, let $h_{i,x}=T_i^{x}-T_i^{x-1}=1/(i+x-1)$, then 
						we can write
						\begin{align}
							\label{NCvuno}
							\bbP[\textswab{N}^{1,i}(T_i^{x})=k+\ell\mid \textswab{N}^{1,i}(T_i^{x-1})=k]=
							\begin{cases}
								\frac{k}{2} h_{i,x}+o(h_{i,x}), & if \quad \ell=1, \\
								o(h_{i,x}), & if\quad \ell= m,\\
								1-\frac{k}{2} h_{i,x}+o(h_{i,x}), & if\quad  \ell=0,\\
								0, & \text{otherwise}.
							\end{cases}
						\end{align}
		        \end{enumerate}
				Observe that the sample paths  of $\{\textswab{N}^{1,i}(T)\}_{T\geq 0}$ and those of $\{\tilde d^1(v_i,n(m+1))\}_{n\geq i}$
				are non-decreasing right-continuous and integer-valued step functions. However, the lengths of the steps in
				$\{\tilde d^1(v_i,n(m+1))\}_{n\geq i}$ always equal unity, while those of $\{\textswab{N}^{1,i}(T)\}_{T\geq 0}$ admit
				the rational values $h_{i,x}$.
				
				Using the well-known relation
				$\sum_{n=1}^M 1/n=\log(M)+\gamma+O(1/M)$, where $\gamma$ is the Euler--Mascheroni constant,
				we have that
				\begin{align*}
					T_i^{x}=\log\Big(1+\frac{x}{(i-1)}\Big)+O(1/i),
				\end{align*}
				so, if $z(i,w_{\ell})\geq1$ for $\ell=1,\dots,b$,
				\begin{align}\label{Tlx} 
					T_i^{z(i,w_{\ell})}=\log\Big(1+\frac{z(i,w_{\ell})}{(i-1)}\Big)+O(1/i)\rightarrow \log (1+c(w_{\ell})),
				\end{align}
				as $i\rightarrow\infty$.

				Analogously, we introduce the times  $\bbT_i^y=\sum_{n=i}^{i+y-1}1/(n+1)$, $y\geq 1$, and $\bbT_i^0=0$. 
				We divide $(0,\bbT_i^y]$ into $y$ disjoint subintervals of lengths $h_{i,y}^*=1/(i+y)$, 
				\begin{align*}
					(0,\bbT_i^y]=(0,\bbT_i^1]\cup(\bbT_i^1,\bbT_i^2]\cup\dots\cup(\bbT_i^{y-1},\bbT_i^{y}].
				\end{align*}
				We introduce the point
				process $\{\textswab{N}^{2,i}(T)\}_{T\geq0}$, jumping at times $\bbT_i^y$, $y\geq1$, and determined by the following properties:
				\begin{enumerate} 
					\item This process starts with an initial random number of individuals supported on $\{m,m+1,\dots,2m\}$ and such that
						\begin{equation}\label{initialcond2}
							\bbP(\textswab{N}^{2,i}(0)\ne m)=1-\prod_{\ell=2}^{m+1}\Big(1-\frac{1}{2(mi+\ell-1)}\Big)=O(1/i).
						\end{equation}
					\item  Its transition probabilities  coincide with
						(\ref{pijk}) when $n=i+y-1$ and $x_n=k$, for every fixed $k \in \mathbb{N}^*$, $k \ge m$, and $y\geq 1$. Hence, for each $y\geq 1$, $h_{i,y}^*=\bbT_i^{y}-\bbT_i^{y-1}=1/(i+y)$, we write
						\begin{align}
							\label{NCvdos}
							\bbP[\textswab{N}^{2,i}(\mathcal{T}_i^{y})=k+\ell\mid \textswab{N}^{2,i}(\mathcal{T}_i^{y-1})=k]=
							\begin{cases}
								\frac{k}{2} h_{i,y}^*+o(h_{i,y}^*), & \ell=1, \\
								o(h_{i,y}^*), & \ell= 2,\\
								1-\frac{k}{2} h_{i,y}^*+o(h_{i,y}^*), & \ell=0,\\
								0, & \text{otherwise}.
							\end{cases}
						\end{align}
				\end{enumerate}
				Then we get that			
				$\bbT_i^{y}=\log(1+y/i)+O(1/i)$.
				Therefore, if $z(i,w_{\ell})\geq1$ for $\ell=1,\dots,b$,
				\begin{align}\label{Tly}
					\bbT_i^{z(i,w_{\ell})}=\log\Big(1+\frac{z(i,w_{\ell})}{i}\Big)+O(1/i)\rightarrow \log (1+c(w_{\ell})),
				\end{align}
				as $i\rightarrow\infty$.

				Note that $\textswab{N}^{1,i}(T_i^x)$ and $\textswab{N}^{2,i}(\mathcal{T}_i^{y})$  have the same law and  initial condition as
				$\tilde d^1(v_i,(i+x)(m+1))$ and $\tilde d^2(v_i,(i+y)(m+1))$, $x,y\geq0$, respectively.
				In addition, by (\ref{initialcond1}) and (\ref{initialcond2}), these processes start with $m$ initial individuals, as $i\rightarrow\infty$.
				
				We emphasize that by (\ref{Tlx}) and (\ref{Tly}) 
				both $T_i^{z(i,w_{\ell})}$ and $\bbT_i^{z(i,w_{\ell})}$ converge to the same time $T_{\ell}=\ln(1+c(w_{\ell}))$, $\ell=1,\dots,b$.
				Moreover, as $i$ increases, (\ref{NCvuno}) and (\ref{NCvdos})
				are closer and closer to the infinitesimal transition probabilities of a Yule process (see (\ref{BirtPro}))
				with intensity $1/2$. 
				Since the transition probabilities and the initial condition determine uniquely the finite-dimensional
				distributions of a Markov process, then as $i\rightarrow\infty$, the finite-dimensional distribution of $\textswab{N}^{1,i}(T_i^{x})$ and $\textswab{N}^{2,i}(\mathcal{T}_i^{y})$, converge to the finite-dimensional distribution of a Yule process with intensity $1/2$. In other words,
				\begin{align}\label{N1Y1}
					\left(\textswab{N}^{1,i}(T_i^{z(i,w_{1})}),\dots,\textswab{N}^{1,i}(T_i^{z(i,w_{b})})\right)&\rightarrow
					\left(N_{1/2}^m(T_1),\dots, N_{1/2}^m(T_b)\right),
				\end{align}
				in distribution, as $i\rightarrow\infty$,  where $N_{1/2}^m(T)$ is the number of individuals of a Yule
				process at time $T$ with intensity $1/2$ and initial population size equal to $m$.
				Analogously,
				\begin{align}\label{N2Y2}
					\left(\textswab{N}^{2,i}(\mathcal{T}_i^{z(i,w_{1})}),\dots,\textswab{N}^{2,i}(\mathcal{T}_i^{z(i,w_{b})})
					\right)&\rightarrow \left(N_{1/2}^m(T_1),\dots, N_{1/2}^m(T_b)\right), 
				\end{align}
				in distribution. To rigorously prove (\ref{N1Y1}) and (\ref{N2Y2}) it is enough to focus on (\ref{N1Y1}) only, as (\ref{N2Y2}) follows in a similar way.
				Recall that $\{\textswab{N}^{1,i}(T)\}_{T\geq0}$ starts with an initial random integer number of individuals $R \in \{m,m+1,\dots,2m\}$. Let $\{\textswab{N}_m^{1,i}(T)\}_{T\geq0}$ denote the  process $\{\textswab{N}^{1,i}(T)\}_{T\geq0}$ conditioned to $\textswab{N}^{1,i}(0)=m$. Set $L_{i,m}=0$, while for $k>m$ define
				\begin{equation}\label{Tik}
					L_{i,k}=\min\{T\geq L_{i,k-1} \colon \textswab{N}_m^{1,i}(T)>k-1\}.
				\end{equation}
				Since the jumps of the process $\{\textswab{N}_m^{1,i}(T)\}$ are of size either 1 or $m$, then observe the following: if at time $L_{i,k}$ the jump is of size 1, i.e., if $\textswab{N}_m^{1,i}(L_{i,k})-\textswab{N}_m^{1,i}(L_{i,k-1})=1$, then  $\textswab{N}_m^{1,i}(L_{i,k})=k$. On the other hand, if $\textswab{N}_m^{1,i}(L_{i,k})-\textswab{N}_m^{1,i}(L_{i,k-1})=m$, then it follows that $\textswab{N}_m^{1,i}(L_{i,k})=k+m-1$, and  we have 
				\begin{equation}\label{tikigual}
					L_{i,k}=L_{i,k+1}=\dots=L_{i,k+m-1}.
				\end{equation}
				Let now $U_{i,j}=L_{i,j+1}-L_{i,j}$, $j\geq m$. If (\ref{tikigual}) holds, 
				\begin{equation}\label{wikigual}
					U_{i,k}=U_{i,k+1}=\dots=U_{i,k+m-1}=0.
				\end{equation}
				Note that we can write 
				\begin{equation}\label{tikUij}
					L_{i,k}=\sum_{j=m}^{k-1}U_{i,j}, \qquad k>m.
				\end{equation}
				In addition  if we consider only the times $L_{i,k}$,  such that $L_{i,k}<L_{i,k+1}$, $k\geq m$, then we can reconstruct $\textswab{N}_m^{1,i}(T)$ for every $T\geq0$:
				\begin{equation}\label{Ntikreconst}
					\textswab{N}_m^{1,i}(T)=k, \quad\text{ if }\quad L_{i,k}\leq T < L_{i,k+1}.
				\end{equation}
				Now we are going to write the finite-dimensional distributions of the original process $\{\textswab{N}^{1,i}(T)\}_{T\geq0}$. Consider the times $0=T_0<T_1< T_2<\dots < T_b$, where $T_{\ell}=\ln(1+c(w_{\ell})$, $\ell=1,\dots,b$, and let $c\in \bbR^b$. Taking into account the initial position of the process, we see that the random vector $\left(\textswab{N}^{1,i}(T_1),\dots,\textswab{N}^{1,i}(T_b)\right)$ has, over $\bbR^b$, the joint distribution 
				\begin{align}\label{5.39}
					\bbP[\left(\textswab{N}^{1,i}(T_1),\dots,\textswab{N}^{1,i}(T_b)\right)= c]=\bbP[\left(\textswab{N}_m^{1,i}(T_1),\dots,\textswab{N}_m^{1,i}(T_b)\right)= c]+\varepsilon_1,
				\end{align}
				where $\varepsilon_1\leq O(1/i)$ by (\ref{initialcond1}).
				Let $m=k_0\leq k_1\leq\dots\leq k_b\in\bbN^*$. By the Markov property we obtain 
				\begin{align}\label{5.40}
					\bbP[\textswab{N}_m^{1,i}(T_{\ell})=k_{\ell},\ell=1,\dots,b]=\prod_{\ell=1}^b \bbP[\textswab{N}_m^{1,i}(T_{\ell})=k_{\ell}\mid \textswab{N}_m^{1,i}(T_{\ell-1})=k_{\ell-1}],
				\end{align}
				where $k_0=m$.  
				Observe that by (\ref{tikUij}) and (\ref{Ntikreconst}), we can write the conditional probabilities $\bbP[\textswab{N}_m^{1,i}(T_{\ell})=k_{\ell}\mid \textswab{N}_m^{1,i}(T_{\ell-1})=k_{\ell-1}]$ as follows. For $\ell=1$ we have
				\begin{align}\label{5.41}
					\bbP[\textswab{N}_m^{1,i}(T_{1})=k_{1}\mid \textswab{N}_m^{1,i}(0)=m]
					& = \bbP[L_{i,k_1}\leq T_1]-\bbP[L_{i,k_1+1}\leq T_1]\\
					& = \bbP\left[\sum_{j=m}^{k_1-1}U_{i,j}\leq T_1\right]-\bbP\left[\sum_{j=m}^{k_1}U_{i,j}\leq T_1\right],\notag
				\end{align}
				while for $\ell=2,\dots,b$,
				\begin{align}
					\bbP[\textswab{N}_m^{1,i}&(T_{\ell})=k_{\ell}\mid \textswab{N}_m^{1,i}(T_{\ell-1})=k_{\ell-1}] \\
					= {} &\bbP[L_{i,k_{\ell}}\leq T_{\ell}<L_{i,k_{\ell}+1}\mid L_{i,k_{\ell-1}}\leq T_{\ell-1}<L_{i,k_{\ell-1}+1}]\nonumber\\
					= {} &\bbP[L_{i,k_{\ell}}-L_{i,k_{\ell-1}+1}\leq T_{\ell}-T_{\ell-1}<L_{i,k_{\ell}+1}-L_{i,k_{\ell-1}}\mid \textswab{N}_m^{1,i}(T_{\ell-1})=k_{\ell-1}]\nonumber\\
					= {} & \bbP[L_{i,k_{\ell}}-L_{i,k_{\ell-1}+1}\leq T_{\ell}-T_{\ell-1}\mid \textswab{N}_m^{1,i}(T_{\ell-1})=k_{\ell-1}] \notag \\
					& -\bbP[L_{i,k_{\ell}+1}-L_{i,k_{\ell-1}}\leq T_{\ell}-T_{\ell-1}\mid \textswab{N}_m^{1,i}(T_{\ell-1})=k_{\ell-1}]\nonumber\\
					= {} &\bbP\left[\sum_{j=k_{\ell-1}+1}^{k_{\ell}-1}U_{i,j}\leq T_{\ell}-T_{\ell-1}\mid \textswab{N}_m^{1,i}(T_{\ell-1})=k_{\ell-1}\right] \notag \\
					& -\bbP\left[\sum_{j=k_{\ell-1}}^{k_{\ell}}U_{i,j}\leq T_{\ell}-T_{\ell-1}\mid \textswab{N}_m^{1,i}(T_{\ell-1})=k_{\ell-1}\right]. \notag
				\end{align}
				Observe also that if $\textswab{N}_m^{1,i}(T_{\ell-1})=k_{\ell-1}$, $\ell\geq1$,  then the random variable $U_{i,k_{\ell-1}}$ is strictly positive, while for $j>k_{\ell-1}$, $U_{i,j}\geq 0$. 
				We focus on the limit distribution only of $L_{i,k_1}=\sum_{j=m}^{k_1-1}U_{i,j}$ as the others follow similarly. 
				Recall that the process $\{\textswab{N}^{1,i}(T)\}_{T\geq0}$ jumps at times of the form $T_i^x=\sum_{n=1}^{i+x-1}1/n$, $x\geq1$, and let $z(i,w_1)$ be such that $\lim_{i\rightarrow\infty} T_i^{z(i,w_1)}=T_1=\ln(1+c(w_1))$, and hence   
				$|T_1-T_i^{z(i,w_1)}|<1/(i+z(i,w_1))$, for $i$ large enough.  Consider the interval $(0,T_1]\times\dots\times(0,T_1]$
				and the partition $(0,T_i^{z(i,w_1)}]\times\dots\times(0,T_i^{z(i,w_1)}]$ 
				in $\bbR^{k_1-m}$. Then  
				\begin{align}\label{jointdist}
					\bbP\left[\sum_{j=m}^{k_1-1}U_{i,j}\leq T_1\right]
					&\sim\bbP[(U_{i,m},\dots,U_{i,k_1-1})\in B_{T_i^{z(i,w_1)}}]\\
					&=\sum_{B_{T_i^{z(i,w_1)}}}p(u_m,\dots,u_{k_1-1}), \notag
				\end{align}
				where $B_{T_i^{z(i,w_1)}}=\{(u_m,\dots,u_{k_1-1}):u_m,+\dots,+u_{k_1-1}\leq T_i^{z(i,w_1)}\}$ and $p(u_m,\dots,u_{k_1-1})$ is the joint probability function of $(U_{m},\dots,U_{i,k_1-1})$.
				By conditioning, the right side of (\ref{jointdist}) can be written as  
				\begin{align}\label{5.44}
					\sum_{u_m=1}^{z(i,w_1)}\sum_{u_{m+1}=u_m}^{z(i,w_1)}&\dots\sum_{u_{k_1-2}=u_{k_1-3}}^{z(i,w_1)}\bbP\left[U_{i,m}=T_i^{u_m}\right]\bbP\left[U_{i,m+1}=T_i^{u_{m+1}}-T_i^{u_m}\mid U_{i,m}=T_i^{u_m}\right]\\
					&\times \dots \times\bbP\left[U_{i,k_1-2}=T_i^{u_{k_1-2}}-T_i^{u_{k_1-3}}\mid \sum_{j=m}^{k_1-3}U_{i,j}=T_i^{k_1-3}\right]\nonumber\\
					&\times \bbP\left[U_{i,k_1-1
					}\leq T_i^{z(i,w_1)}-T_i^{u_{k_1-2}}\mid \sum_{j=m}^{k_1-2}U_{i,j}=T_i^{k_1-2}\right]. \notag
				\end{align}
				By (\ref{NCvuno}) and since the event $\{U_{i,m}=T_i^{u_m}\}$ means that at times $T_i^1,\dots,T_i^{u_m-1}$, the process $\textswab{N}_m^{1,i}(T)$ did not jump, but at time $T_i^{u_m}$ it did, 
				\begin{align}
					\label{5.45}
					\bbP[U_{i,m}=T_i^{u_m}] = & \left[\prod_{\ell=1}^{u_m-1}  \left(
					1-\frac{m}{2}h_{i,\ell}+o(h_{i,\ell})\right)\right]
					\left(\frac{m}{2}h_{i,u_m}+o(h_{i,u_m})\right)\\
					& = \frac{m}{2}h_{i,u_m}\exp\left(\sum_{\ell=1}^{u_m-1}\ln\left(1-\frac{m}{2}
					h_{i,\ell}+o(h_{i,\ell})\right)\right)+o(h_{i,u_m}). \notag
				\end{align}
				Using Taylor expansion for $\ln(1-y)$ and $e^y$, and rearranging the terms we obtain that
				\begin{align}\label{5.46}
					\bbP[U_{i,m}=T_i^{u_m}] =\frac{m}{2}h_{i,u_m}\exp\left(-\frac{m}{2}T_i^{u_m}\right)+\text{Err}_1,
				\end{align}
				where $\text{Err}_1\leq O(h_{i,u_m}^2)$.
				Next, we calculate $\bbP[U_{i,j}=T_i^{u_j}-T_i^{u_{j-1}}\mid L_{i,j}=T_i^{u_{j-1}}]$, $j=m+1,\dots,k_1-1.$ Let $J_{i,j}$ be the size of the jump at time $L_{i,j}$. By (\ref{tikigual}) and (\ref{wikigual}), if $J_{i,j}>1$, then $U_{i,j}=0$, otherwise, if $J_{i,j}=1$ then $U_{i,j}>0$.  In addition, by (\ref{NCvuno}) and (\ref{Tik}) we have that for every $x\geq 1$,
				\begin{align}\label{Jij}
					\bbP[J_{i,j}>1\mid L_{i,j}=T_i^x] & = \bbP[\textswab{N}_m^{1,i}(T_i^x)=j+m-1\mid \textswab{N}_m^{1,i}(T_i^{x-1})=j-1] =  o(h_{i,x}).
				\end{align}
				Thus, conditioning on $J_{i,j}$ and using (\ref{Jij}) we get
				\begin{align}\label{5.48}
					\bbP[U_{i,j}=T_i^{u_j}&-T_i^{u_{j-1}}\mid L_{i,j}=T_i^{u_{j-1}}]  \\
					& = \bbP[U_{i,j}=T_i^{u_j}-T_i^{u_{j-1}}\mid L_{i,j}=T_i^{u_{j-1}}, J_{i,j}=1] + \text{Err}_2, \notag
				\end{align}
				where $\text{Err}_2\leq o(h_{i,u_j})$. Observe that the conditional event $\{ U_{i,j}=T_i^{u_j}-T_i^{u_{j-1}}\mid L_{i,j}=T_i^{u_{j-1}}, J_{i,j}=1\}$ means that at times $T_i^{u_{j-1}+1},T_i^{u_{j-1}+2},\dots,T_i^{u_{j}-1}$, the process $\textswab{N}_m^{1,i}(T)$ did not jump, but at time $T_i^{u_{j}}$ it did. Therefore,  by (\ref{NCvuno}) and applying similar arguments to those used to obtain  (\ref{5.45}) and (\ref{5.46}) we have
				\begin{align}\label{5.49}
					\bbP[U_{i,j}=T_i^{u_j}-T_i^{u_{j-1}}&\mid L_{i,j}=T_i^{u_{j-1}}, J_{i,j}=1]\\
					&=\left[\prod_{\ell=u_{j-1}+1}^{u_j-1}\left(1-\frac{j}{2}h_{i,\ell}+o(h_{i,\ell})\right)\right]\left(\frac{j}{2}h_{i,u_j}+o(h_{i,u_j})\right)\nonumber\\
					&=\frac{j}{2}h_{i,u_j} \exp\left[-\frac{j}{2}(T_i^{u_j}-T_i^{u_{j-1}})\right]+\text{Err}_3, \notag
				\end{align}
				where $\text{Err}_3\leq O(h_{i,u_j}^2)$. Thus, by (\ref{5.48}) and (\ref{5.49})
				\begin{align}\label{5.50}
					\bbP[U_{i,j}=T_i^{u_j}-T_i^{u_{j-1}}\mid L_{i,j}=T_i^{u_{j-1}}]=\frac{j}{2}h_{i,u_j} \exp\left[-\frac{j}{2}(T_i^{u_j}-T_i^{u_{j-1}})\right]+\text{Err}_4,  
				\end{align}	
				where $\text{Err}_4\leq O(h_{i,u_j}^2)$.
				Following these same steps we can also find that
				\begin{align}\label{5.51}
					\bbP[U_{i,j}>T_i^{u_j}-T_i^{u_{j-1}}\mid L_{i,j}=T_i^{u_{j-1}}]= \exp\left[-\frac{j}{2}(T_i^{u_j}-T_i^{u_{j-1}})\right]+\text{Err}_5,  
				\end{align}
				where $\text{Err}_5\leq O(h_{i,u_j}^2)$.
				Now, by substituting (\ref{5.46}), (\ref{5.50}) and (\ref{5.51}) in (\ref{5.44}) we arrive at
				\begin{align}\label{riemman}
					\bbP\left[\sum_{j=m}^{k_1-1}U_{i,j}\leq T_1\right]&\sim \sum_{u_m=1}^{z(i,w_1)}\sum_{u_{m+1}=1}^{z(i,w_1)}\dots\sum_{u_{k_1-2}=1}^{z(i,w_1)}h_{i,u_m}h_{i,u_{m+1}}\frac{m}{2}\exp\left[-\frac{m}{2}T_i^{u_m}\right]\\
					&\times \frac{(m+1)}{2}\exp\left[\frac{-(m+1)}{2}(T_i^{u_{m+1}}-T_i^{u_m})\right]\mathbbm{1}_{\{T_i^{u_{m+1}} \geq T_i^{u_m}\}}\nonumber\\
					&\times \dots \times h_{i,u_{k_1-2}}\frac{(k_1-2)}{2}\exp\left[\frac{-(k_1-2)}{2}(T_i^{u_{k_1-2}}-T_i^{u_{k_1-3}})\right]\nonumber\\
					&\times \left[1-\exp\left(\frac{-(k_1-1)}{2}(T_i^{z_{i,w_1}}-T_i^{u_{k_1-2})}\right)\right] \mathbbm{1}_{\{T_i^{z_{i,w_1}} \geq T_i^{u_{k_1-2}}\}}
					+\text{Err}_6, \notag
				\end{align}
				where $\text{Err}_6\leq O\left(h_{i,z(i,w_1)}\prod_{\ell=m}^{k_1-2}h_{i,\ell}\right)$.
				Taking the limit as $i\rightarrow\infty$, the right-hand side of (\ref{riemman}) becomes
				\begin{align}\label{5.53}
					\int_0^{T_1}\int_{y_{m+1}}^{T_1}\dots\int_{y_{k_1-2}}^{T_1}&\frac{m}{2}\frac{(m+1)}{2}\exp\left(-\frac{m}{2}y_m\right)
					\exp\left(-\frac{m+1}{2}(y_{m+1}-y_m)\right)\\
					&\times \dots \times \frac{(k_1-2)}{2}\exp\left(-\frac{k_1-2}{2}(y_{k_1-2}-y_{k_1-3})\right)\nonumber\\
					&\times \left[1-\exp\left(-\frac{k_1-1}{2}(T_1-y_{k_1-2})\right)\right]dy_{k_1-2}\dots dy_m, \notag
				\end{align}
				as the right side of (\ref{riemman}) is the Riemann sum for the above definite integral. We observe that this integral also corresponds to $\bbP[Y_m+\dots +Y_{k_1-1}\leq T_1]$, where the random variables $Y_j$, $j=1,\dots,k_1-1$, are independent and exponentially  distributed with parameter $j/2$, respectively.   
				Therefore, 
				\begin{equation}
					\lim_{i\rightarrow\infty}\bbP\left[\sum_{j=m}^{k_1-1}U_{i,j}\leq T_1\right]=\bbP[Y_m+\dots +Y_{k_1-1}\leq T_1],
				\end{equation}
				and in an analogous way we also obtain
				\begin{equation}
					\lim_{i\rightarrow\infty}\bbP\left[\sum_{j=m}^{k_1}U_{i,j}\leq T_1\right]=\bbP[Y_m+\dots +Y_{k_1}\leq T_1].
			 	\end{equation}
				From this and (\ref{5.41}) we get
				\begin{align}\label{5.54}
					\lim_{i\rightarrow\infty}\bbP[\textswab{N}_{1,i}(T_{1})=k_{1}\mid \textswab{N}_{1,i}(0)=m]&=\lim_{i\rightarrow\infty}\bbP[\textswab{N}_m^{1,i}(T_{1})=k_{1}]\\
					&=\bbP\left[\sum_{j=m}^{k_1-1}Y_j\leq T_1\right]-\bbP\left[\sum_{j=m}^{k_1}Y_j\leq T_1\right]\nonumber\\
					&=\bbP\left[\sum_{j=m}^{k_1-1}Y_j\leq T_1<\sum_{j=m}^{k_1}Y_j\right]. \notag
				\end{align}
			 	In a Yule process with parameter $\lambda$ and starting with $m$ initial individuals, the interarrival or sojourn times are independent random variables exponentially distributed with parameter $\lambda j$, $j\geq m$, respectively. Thus, (\ref{5.54}) corresponds to
				\begin{align}\label{5.55}
			 		\bbP[N_{1/2}^m(T_1)=k_1\mid N_{1/2}^m(0)=m],
				\end{align}
			 	where $N_{1/2}^m(T)$ is the number of individuals in a Yule process with parameter $1/2$ and starting with $m$ initial individuals. 
			 	Following analogous steps from (\ref{5.41}) to (\ref{5.55}) we also find that for $\ell=2,\dots,b$,
			 	\begin{align}\label{5.56}
		     		\lim_{i\rightarrow\infty}\bbP[\textswab{N}_{1,i}(T_{\ell})=k_{\ell}\mid \textswab{N}_{1,i}(T_{\ell-1})=k_{\ell-1}]=\bbP[N_{1/2}^m(T_{\ell})=k_{\ell}\mid N_{1/2}^m(T_{\ell-1})=k_{\ell-1}].
			 	\end{align}
			 	Consequently, from  (\ref{5.39}), (\ref{5.40}), (\ref{5.54}), (\ref{5.55}), (\ref{5.56}), and since the Yule process is Markov, we conclude that
			 	\begin{align}
					\lim_{n\rightarrow\infty}\bbP[\textswab{N}_m^{1,i}(T_{\ell})=k_{\ell},\ell=1,\dots,b]&=\prod_{\ell=1}^b \bbP[N_{1/2}^m(T_{\ell})=k_{\ell}\mid N_{1/2}^m(T_{\ell-1})=k_{\ell-1}]\\
					&=\bbP[N_{1/2}^m(T_{\ell})=k_{\ell}, \ell=1,\dots,b]. \notag
				\end{align}

				Now, once proven the convergence to the Yule process of intensity $1/2$, we immediately have that
				\begin{align}
					\label{d1N}
					\left(\tilde d^1(v_i,(i+z(i,w_{\ell}))(m+1)),\ell=1,\dots,b\right)&\rightarrow
					\left(N_{1/2}^m(T_{\ell}),\ell=1,\dots,b\right),
				\end{align}
				and
				\begin{align}\label{d2N}
					\left(\tilde d^2(v_i,(i+z(i,w_{\ell}))(m+1)),\ell=1,\dots,b\right)&\rightarrow
					\left(N_{1/2}^m(T_{\ell}),\ell=1,\dots,b\right), 
				\end{align}	
				in distribution, as $i\rightarrow\infty$.
				
				Observe that at time $n(m+1)$, $n \ge i$, by (\ref{coupling}) the random variables $\tilde d^1(v_i,n(m+1))$,
				$\tilde d(v_i,n(m+1))$ and $\tilde d^2(v_i,n(m+1))$ are almost surely ordered, that is
				\begin{align}
					\label{d1d}
					\bbP[\tilde d^2(v_i,(i+z(i,w_{\ell}))(m+1)) &\leq \tilde d(v_i,(i+z(i,w_{\ell}))(m+1))\\
				 	&\leq \tilde d^1(v_i,(i+z(i,w_{\ell}))(m+1)), \ell=1,\dots,b]=1.\nonumber
				\end{align}
				This implies that for $k\geq m$,
				\begin{align}\label{distr}
					\bbP & (\tilde d^1(v_i,(i+z(i,w_{\ell}))(m+1))\leq k_{\ell},\ell=1,\dots,b)\\
					&\leq \bbP(\tilde d(v_i,(i+z(i,w_{\ell}))(m+1))\leq k_{\ell},\ell=1,\dots,b)\nonumber\\
					&\leq \bbP(\tilde d^2(v_i,(i+z(i,w_{\ell}))(m+1))\leq k_{\ell},\ell=1,\dots,b).\notag
				\end{align}
				Thus, from (\ref{d1N}), (\ref{d2N}) and (\ref{distr}) we obtain the convergence in distribution of the random vector
				\begin{equation} 
					\label{Margd}
					\left(\tilde d(v_i,(i+z(i,w_{\ell}))(m+1)),\ell=1,\dots,b\right)\rightarrow \left(N_{1/2}^m,
					(T_{\ell}),\ell=1\dots,b\right), 
				\end{equation}
				as $i\rightarrow\infty$. 
			\end{proof} 
		
			\begin{proof}[\textbf{Proof of Theorem \ref{Teo1a}}] 
				Using Lemma \ref{lema1}, Corollary \ref{cor2} and Lemma \ref{lema1conv} we obtain the convergence
				to the $b$-finite-dimensional distributions of a Yule process, for all $b\geq1$. 
				To obtain the exact formula we make use of the independence of the increments and of the distribution of
				the number of individuals in a Yule process with $k_{\ell}$ initial progenitors, $\ell=0,\dots, b$. Thus,
				\begin{align*}
					\bbP[ & N_{1/2}^m(\log(1+c(w_1)))=k_1,\dots,N_{1/2}^m(\log(1+c(w_b)))=k_b]\\
					& = \prod_{\ell=1}^b \bbP\left(N_{1/2}^{k_{\ell-1}}\left(\log\left(\frac{1+c(w_{\ell})}{1+c(w_{\ell-1})}\right)
					\right)=k_{\ell}\right). 
				\end{align*}
				Finally, we use equation (3.5) in \cite{Feller1}, Section XVII.3.
			\end{proof}

		\subsection{A lemma and the proof of Theorem \ref{Teo1} }\label{process2}

			In this section we make use of the planted model described in Section \ref{monotone}. Notice that
			the BA random graph model corresponds to the case in which all the random variables $M_j$, $j\geq1$,
			are concentrated on $m$, so that $\mathfrak{T}_n = n(m+1)$ almost surely.
 
			Formally, let  $({G}_m^t)_{t\geq 1}$ be the random graph process defining the BA model as in subsection \ref{BAl}.
			For each $1\leq j\leq i$ consider the birth processes in discrete time $\{b(v_j,n(m+1))\}_{n\geq i}$,
			with state space given by $\bbN^*$ and determined by the transition probabilities 
            \begin{align}
               	\label{b(n)}
               	\bbP[b(v_{j},(n+1)(m+1)=k+\ell\mid b(v_{j},n(m+1))=k]=
                \begin{cases}
					\frac{k}{n}, & \ell=1, \\
					0, & \text{otherwise},
				\end{cases}
            \end{align} 
			and initial condition $b(v_{j},i(m+1))=1$ almost surely.
			Recall that $i$ is taken large so that Lemma \ref{lema1} holds.
			\begin{lem}
				\label{S-Yule2}
				Let $z(i,w):\bbN^*\times\bbR^+\rightarrow\bbN$ be a function such that
				$c(w):=\lim_{i\rightarrow\infty}z(i,w)/i$ exists finite, where $c(w):\bbR^+\rightarrow \bbR^+$ is an increasing
				function in $w$, and let  $w_1<\dots<w_b$, $b\in\bbN^*$, be positive real numbers. For every $1 \leq j \leq i$ we have
				\begin{align}
					\label{Vertices}
					( b(v_j,(i+z(i,w_{\ell}))(m+1)),\ell=1,\dots,b)
					\rightarrow(N_{1}^1(\log(1+c(w_{\ell}))),\ell=1,\dots,b)
				\end{align}
				in distribution as $i\rightarrow\infty$, where $N_{1}^1(T)$ is the number of individuals of a Yule process at time $T$, with
				one initial individual and parameter $1$. 
			\end{lem}

			\begin{proof}		
				For each $1 \leq j \leq i$,  we prove convergence in distribution in the same way as we did in the proof
				of Lemma \ref{lema1conv} for $\textswab{N}^{1,i}(T_i^x)$, but now with $\textswab{N}^1(0)=1$,
				i.e.\ the process starts with only one individual, and transition probabilities given by (\ref{b(n)}).
				Therefore, as $i\rightarrow\infty$, the probabilities (\ref{b(n)}) become the infinitesimal transition probabilities
				of a Yule process with intensity $1$, starting with one individual.
				Since the process is Markov, the transition probabilities and the initial condition determine
				uniquely the finite-dimensional distributions.  
			\end{proof}
			
			\begin{obs}		
				To prove the first part of Theorem \ref{Teo1} we will make use of the result of Theorem \ref{Teo1a}. The idea is to  take $n:=n(i,w)$,
				a function of $i$ and a positive real number $w$, such that,  $i/n(i,w)\rightarrow 1/(1+c(w))$ as $i\rightarrow\infty$, with $c(w)$
				as in Theorem \ref{Teo1a}. Thus, $\lim_{w\rightarrow\infty}\lim_{i\rightarrow\infty}i/n(i,w)=0$. 
			\end{obs}

			\begin{proof}[\textbf{Proof of Theorem \ref{Teo1}}]
				We start proving (\ref{uniformchoice}). Consider the BA model at time $t=n(m+1)$, $n\geq i$,
				and the planted model of Section \ref{monotone}. Recall that in the planted model we have $i$ discrete-time birth
				processes $\{b(v_j,n(m+1))\}_{n\geq i}$, $j=1,\dots,i$, which are exchangeable.
				By Theorem \ref{plantedinduce},
				the event of choosing a vertex uniformly at random in the BA model is equivalent to
				that of selecting first uniformly at random one of the $i$ processes $\{b(v_j,n(m+1))\}_{n\geq i}$, $j=1,\dots,i$,
				and then choosing uniformly at random a vertex belonging to it. Therefore, the degree of $V_t$ can be studied through the analysis
				of the degree of a random vertex chosen with uniform probability  between the vertices in any of the $i$ processes
				$\{b(v_j,n(m+1))\}_{n\geq i}$, $j=1,\dots,i$. Let $V_{t}^j$ be a vertex chosen uniformly at random from the vertices
				in the $j$-th process $\{b(v_j,n(m+1))\}_{n\geq i}$, and let $\epsilon(i,n)$ be a function we will use to measure the error.
				Using the notation of Section \ref{samplingPlanted}, where $W$ denotes the index of the birth process chosen and
				$Y_j$ is a random variable taking values in $\{1,2,\dots,n-i+1\}$ denoting the number of vertices in $b(v_j,n(m+1))$,
				we have 
				\begin{align}\label{Vt1}				
					\bbP[d(V_t)=k] ={}& \sum_{j=1}^i\bbP[d(V_t^j)=k,V_t^j\neq v_j, W=j] \\
					& +\sum_{j=1}^i\bbP[d(V_t^j)=k,V_t^j= v_j, W=j]\nonumber\\
					={}& \bbP[d(V_t^1)=k\mid V_t^1\neq v_1, W=1]\sum_{j=1}^i\bbP[V_t^j\neq v_j, W=j]\nonumber\\
					& + \sum_{j=1}^i\bbP[d(V_t^j)=k\mid V_t^j= v_j, W=j]\bbP[V_t^j=v_j, W=j]\nonumber\\
					={}& \bbP[d(V_t^1)=k\mid V_t^1\neq v_1,W=1]+\epsilon(i,n).\notag
				\end{align} 
				The last two equalities are obtained by considering the following two observations.
				First, permuting the labels of the $i$ birth processes $\{b(v_j,n(m+1))\}_{n\geq i}$, $j=1,\dots,i$, will not change
				the distribution of the process of the new vertices and their degrees, thus for $j=1,\dots,i$,
				we can write $\bbP[d(V_t^j)=k\mid V_t^j\neq v_j, W=j]=\bbP[d(V_t^1)=k\mid V_t^1\neq v_1, W=1]$. Second, 
				\begin{align}
					\sum_{j=1}^i\bbP[d(V_t^j)=k\mid V_t^j= v_j, W=j]\bbP[V_t^j=v_j, W=j] & \leq \sum_{j=1}^i\bbP[V_t^j=v_j,W=j]\\
					&=\sum_{j=1}^i\sum_{\ell=1}^{n-i+1}\frac{1}{\ell}\frac{\ell}{n}\bbP(Y_j=\ell)
					=\frac{i}{n}, \notag
				\end{align}
				that is, $\epsilon(i,n) = O(i/n)$.
				
				Note that the degree of the planted vertices behaves differently as they have appeared in the very early history of the
				graph evolution. Also, in the limit, the number of planted vertices becomes negligible compared to the total size of the graph.
				   
				Now take $n(i,w)=i+z(i,w)$, where $z(i,w)$ is defined as in Lemma \ref{lema1conv}, Lemma \ref{S-Yule2} and Theorem \ref{Teo1a}.
				As $i\rightarrow\infty$,		
				\begin{itemize}
					\item by Lemma \ref{S-Yule2} we have that $b(v_1,n(m+1))$, converges
						in distribution to the size of a Yule process evaluated at time $T=\log(1+c(w))$, with intensity $1$ and
						starting with one initial individual;
					\item by Lemma \ref{lema1conv},  the degree of each vertex  belonging to
					$\{b(v_1,n(m+1))\}$, given that it is different to $v_1$,
						converges in distribution
						to the size of a Yule process with intensity $1/2$ and $m$ initial individuals. 
				\end{itemize}
				The above Yule processes describe an $m$-Yule model $\{Y_{1/2,1}^m(T)\}_{T\geq0}$ of parameters $\lambda=1/2$ and $\beta=1$.
				For $i\rightarrow\infty$, the degree of $V_{t}^1$ given that $V_{t}^1\neq v_1$, converges in distribution
				to the size of a genus chosen uniformly at random in the $m$-Yule model at time $T=\log(1+c(w))$, given in turn
				that such a random genus is different to the first genus appeared, $g_1$.
				Thus, if $\mathcal{N}_T^m$ denotes the size of a genus $G_T$ chosen uniformly at random at time $T$ in $\{Y_{1/2,1}^m(T)\}$, 
                \begin{equation}
					\label{Vt1Ng1}
					\lim_{i\rightarrow\infty}\bbP(d(V_{t}^1)=k\mid V_{t}^1\neq v_1, W=1)=\bbP(\mathcal{N}_{\log(1+c(w))}^m
					=k\mid G_{\log(1+c(w))}\neq g_1).
				\end{equation} 
				By (\ref{Vt1}) and (\ref{Vt1Ng1}),
				\begin{equation}
					\label{Vt1Ng11}
					\lim_{i\rightarrow\infty}\bbP(d(V_{t})=k)=\bbP(\mathcal{N}_{\log(1+c(w))}^m=k\mid G_{\log(1+c(w))}\neq g_1)+\varepsilon(w),	
				\end{equation}
				where $\varepsilon(w) = O\left(1/(1+c(w))\right)$. Since $c(w)$ is an increasing function and a Yule process is supercritical, then
               \begin{equation}
					\label{Ng1}
					\lim_{w\rightarrow\infty}\bbP(\mathcal{N}_{\log(1+c(w))}^m=k)
					=\lim_{w\rightarrow\infty}\bbP(\mathcal{N}_{\log(1+c(w))}^m=k\mid G_{\log(1+c(w))}\neq g_1).	
				\end{equation} 							
				Therefore, by (\ref{Vt1Ng11}) and (\ref{Ng1}), 
				\begin{align}
					\lim_{w\rightarrow\infty}\lim_{i\rightarrow\infty}\bbP(d(V_{t})=k)=
					\lim_{w\rightarrow\infty}\bbP(\mathcal{N}_{\log(1+c(w))}^m=k).
				\end{align}

				To prove (\ref{Inprobability}) note that
				\begin{align*}
					\bbE\Big(\frac{1}{n}\sum_{i=1}^n \mathbb{I}_{\{d(v_i,t)=k\}}\Big)=\frac{\bbE N_{k,t}}{n}=
					\frac{1}{n}\sum_{i=1}^n\bbP(d(v_i,t)=k)=\bbP(d(V_t)=k).
				\end{align*}
				Let $\bbF_t$ be the natural filtration generated by the process
				$\{N_{k,t}\}_{t\geq1}$ up to time $t$, and
				define $Z_s=\bbE(N_{k,t}\mid\bbF_s )$. Observe that $Z_s$ is a martingale as
				$\bbE[\bbE(N_{k,t}\mid\bbF_s )\mid \bbF_r]=\bbE(N_{k,t}\mid\bbF_r )$,
				for $r\leq s\leq t$.
				Considering that  at each time interval $(s-1,s]$ a new vertex $v_s$ appears and $m$ directed edges
				from it are attached to existing vertices,
				then $v_s$ is attached to at most $m$ different vertices, say $v^1,\dots,v^m$. This does not affect neither the degree of the other existing vertices $w\neq v^1,\dots,v^m$, nor the attachment probabilities related to them.
				Thus, it follows that $|Z_s-Z_{s-1}|\leq 2m$. Since $Z_t=N_{k,t}$ and $Z_0=\bbE N_{k,t}$,
				then by taking $x=C\sqrt{t\log t}$, with $C>m\sqrt{8}$
				and applying the Azuma--Hoeffding inequality (see Lemma 4.1.3 in \cite{Durrett2006}), we obtain
				\begin{align}\label{ByAzuma}
					\bbP\Big(\Big|\frac{N_{k,t}}{n}-\frac{\bbE N_{k,t}}{n}\Big|>C\sqrt{\frac{(m+1)\log(n(m+1))}{n}}\Big)\leq o\Big(\frac{1}{n}\Big).
				\end{align}
				Now observe that $N_{k,t}=0$ when $k\geq n(m+1)$, $n\geq 1$. Therefore, 
				\begin{align}
					\bbP\Big(\max_k & \Big|\frac{N_{k,t}}{n}-\frac{\bbE N_{k,t}}{n}\Big|>C\sqrt{\frac{(m+1)\log(n(m+1))}{n}}\Big)\\
					& = \bbP\Big(\max_{k<n(m+1)} \Big|\frac{N_{k,t}}{n}-\frac{\bbE N_{k,t}}{n}\Big|>C\sqrt{\frac{(m+1)\log(n(m+1))}{n}}\Big)\notag \\
					&\leq \sum_{k=1}^{n(m+1)-1}\bbP\Big(\max_{k<t} \Big|\frac{N_{k,t}}{n}
					-\frac{\bbE N_{k,t}}{n}\Big|>C\sqrt{\frac{(m+1)\log(n(m+1))}{n}}\Big). \notag
				\end{align}
				Concluding, by (\ref{ByAzuma}) we get the desired result.
			\end{proof}

		\subsection{Proof of Proposition \ref{Teo2}}

			\begin{proof}
  				Let us consider an $m$-Yule model $\{Y_{1/2,1}^m(T)\}_{T\geq0}$.
				It is known that by conditioning on the number of genera present
				at time $T$, the random times at which novel genera appear are distributed as
				the order statistics of i.i.d.\ random variables distributed with distribution function given by
				(see e.g.\ \cite{LaskiPolito14} or \cite{MR3562427} and the references therein)
				\begin{equation}
					\label{orderstat}
					\bbP(\mathcal{T}\leq\tau)=\frac{e^{\tau}-1}{e^{T}-1},\qquad 0\leq\tau\leq T.
				\end{equation}
				
				As above, let $\mathcal{N}_T^m$ denote the size of a genus chosen uniformly at random at time $T$. Then,
				for every $k\geq m$ and recalling the distribution of a Yule process starting with $m$ initial individuals,
				\begin{align}
					\label{first}
					\bbP(\mathcal{N}_T^m=k) & = \int_0^T\bbP(N_{1/2}^m(T)=k\mid
					N_{1/2}^m(\tau)=m)\bbP(\mathcal{T}\in \textup{d}\tau)\\
					&= \int_0^T \binom{k-1}{m-1} e^{-m \frac{T-\tau}{2}}(1-e^{\frac{T-\tau}{2}})^{k-m}
					\frac{e^{\tau}}{e^{T}-1}\textup{d}\tau\nonumber\\
					&= \frac{1}{1-e^{-T}}\int_0^T \binom{k-1}{m-1} e^{-y}e^{-m \frac{y}{2}}(1-e^{-\frac{y}{2}})^{k-m} \textup{d}y. \notag
				\end{align}
				By letting $z=1-e^{-\frac{y}{2}}$, we can write (\ref{first}) as
				\begin{align}
					\label{third}
					\bbP(\mathcal{N}_T^m=k)
					&= \frac{2}{1-e^{-T}}\int_0^ {1-e^{- \frac{T}{2}}}\binom{k-1}{m-1} z^{k-m}(1-z)^{m+1} \textup{d}z.
				\end{align} 					
				Our interest is in the asymptotic behaviour when $T\rightarrow\infty$. In this case \eqref{third} reduces to
				\begin{align}
					\lim_{T\rightarrow\infty} \bbP(\mathcal{N}_T^m=k) 
					&= 2\binom{k-1}{m-1}B(k-m+1,m+2) \\
					&= m(m+1)B(k,3), \notag
				\end{align}
				where $B(a,b)$ denotes the Beta function.
			\end{proof}
			
	\subsubsection*{Acknowledgments} 
	
		The authors would like to thank the anonymous Referee for the valuable suggestions which improved both the content and presentation of the paper.
	
		F.\ Polito and L.\ Sacerdote have been supported by the projects \emph{Memory in Evolving Graphs}
		(Compagnia di San Paolo/Universit\`a di Torino) and by INDAM (GNAMPA/GNCS).

\end{document}